%% file: main.tex
\documentclass{article}
\usepackage[numbers,sort&compress]{natbib}
\usepackage[letterpaper,margin=1.0in]{geometry}
\usepackage[utf8]{inputenc} 
\usepackage[T1]{fontenc}    
\usepackage{hyperref}       
\usepackage{url}            
\usepackage{booktabs}       
\usepackage{amsfonts}       
\usepackage{nicefrac}       
\usepackage{microtype}      
\usepackage{xcolor}         
\usepackage{algorithm}
\usepackage{algorithmic}
\usepackage{graphicx}
\usepackage{enumitem}
\input{math.tex}

\title{Near Optimal Stochastic Algorithms for Finite-Sum Unbalanced Convex-Concave Minimax Optimization}
\author{
    Luo Luo \thanks{Equal Contribution} \thanks{
        Department of Mathematics, The Hong Kong University of Science and Technology; luoluo@ust.hk} \qquad
    Guangzeng Xie $^*$\thanks{
        Academy for Advanced Interdisciplinary Studies, Peking University; smsxgz@pku.edu.cn} \qquad
    Tong Zhang \thanks{
        Department of Mathematics, The Hong Kong University of Science and Technology; tongzhang@ust.hk} \qquad
    Zhihua Zhang \thanks{School of Mathematical Sciences, Peking University; zhzhang@math.pku.edu.cn}
}
\date{}

\begin{document}

\maketitle

\begin{abstract}
This paper considers stochastic first-order algorithms for convex-concave minimax problems of the form $\min_{\bf x}\max_{\bf y}f(\bf x, \bf y)$, where $f$ can be presented by the average of $n$ individual components which are $L$-average smooth. 
For $\mu_x$-strongly-convex-$\mu_y$-strongly-concave setting, 
we propose a new stochastic algorithm which could find an $\varepsilon$-saddle point of the problem in $\tilde{\mathcal O} (\sqrt{n(\sqrt{n}+\kappa_x)(\sqrt{n}+\kappa_y)}\log(1/\varepsilon))$ stochastic first-order complexity, where $\kappa_x\triangleq L/\mu_x$ and $\kappa_y\triangleq L/\mu_y$.
This upper bound is near optimal with respect to $\varepsilon$, $n$, $\kappa_x$ and $\kappa_y$ simultaneously. 
In addition, the algorithm is easily implemented and works well in practice.
Our methods can be extended to solve more general unbalanced convex-concave minimax problems and the corresponding upper complexity bounds are also near optimal.
\end{abstract}

\section{Introduction}

This paper studies the following finite-sum minimax problem:
\begin{align}
    \min_{\vx\in\fX}\max_{\vy\in\fY} f(\vx,\vy)\triangleq \frac{1}{n}\sum_{i=1}^n f_i(\vx,\vy), \label{prob:main}
\end{align}
where $\fX$ and $\fY$ are convex and closed. Our goal is to find the approximate saddle point of problem (\ref{prob:main}) to guarantee the duality gap  no larger than $\eps$. 

The formulation (\ref{prob:main}) includes a lot of machine learning applications such as AUC maximization~\cite{hanley1982meaning,ying2016stochastic,liu2020stocastic,guo2020communication}, robust optimization~\cite{duchi2019variance,yan2019stochastic}, adversarial learning~\cite{sinha2017certifying} and reinforcement learning~\cite{wai2018multi}. 
We study the fundamental setting that each component $f_i(\vx,\vy)$ is convex in $x$ and concave in $y$; $\{f_i(\vx,\vy)\}_{i=1}^n$ are $L$-average smooth; and $f(\vx,\vy)$ is $\mu_x$-strongly-convex in $x$ and $\mu_y$-strongly-concave in $y$. In particular, we focus on the unbalanced case that the condition numbers $\kappa_x \triangleq L/\mu_x$ and $\kappa_y\triangleq L/\mu_x$ could be quite different. Without loss of generality, we suppose $\kappa_x \geq \kappa_y$.

Most existing algorithms~\cite{korpelevich1977extragradient,gidel2019variational,palaniappan2016stochastic,luo2019stochastic,chavdarova2019reducing,mokhtari2020unified,alacaoglu2021stochastic} for general strongly-convex-strongly-concave (SCSC) minimax optimization do not consider the difference between two condition numbers $\kappa_x$ and $\kappa_y$, which leads to their upper bound complexities depend on $\max\{\kappa_x,\kappa_y\}$. 
\citet{lin2020near} first proposed proximal point methods for the unbalanced problem with $\tilde\fO(\sqrt{\kappa_x\kappa_y}\log^3(1/\eps))$ gradient calls\footnote{We use notation $\tilde\fO(\cdot)$ to hide logarithmic factors of $n$, $\kappa_x$ and $\kappa_y$ in complexities.}, nearly matching the lower bound~\cite{zhang2019lower,ibrahim2020linear} of the deterministic algorithm.  
\citet{wang2020improved,xie2021dippa} improved~\citet{lin2018catalyst}'s results under refined smoothness assumption and bilinear setting. 
Unfortunately, these methods~\cite{lin2020near,wang2020improved} are based on full gradient oracle and ignore the finite-sum structure in the objective function. 

In practice, the number of components $n$ could be very large. It is natural to use stochastic first-order oracle (SFO) algorithms to reduce the cost of gradient based methods.
The SFO algorithms and their optimality are well-studied for minimization problems~\cite{frostig2015regularizing,johnson2013accelerating,defazio2014saga,allen2017katyusha,fang2018spider,lin2018catalyst,zhou2019lower,woodworth2016tight}, but the related theory for minimax optimization is still imperfect. 
In the balanced case of $\kappa_x=\kappa_y=\kappa$, \citet{palaniappan2016stochastic} first introduced SVRG/SAGA~\cite{johnson2013accelerating,defazio2014saga} to solve the general formulation (\ref{prob:main}) in big data regime and obtained $\tilde\fO((n+\kappa^2)\log(1/\eps))$ SFO upper bound. If $\kappa>\sqrt{n}$, the leading term can be reduced to $\fO(n+\kappa\sqrt{n})$ by introducing extragradient (EG) or proximal point iterations~\cite{luo2019stochastic,chavdarova2019reducing}. Recently, \citet{vladislav2021accelerated} proposed an accelerated stochastic method for unbalanced problem (\ref{prob:main}) when both $\kappa_x$ and $\kappa_y$ are larger than $\sqrt{n}$. They designed a complicated algorithm which contains three-loops iterations,  obtaining $\fO(\sqrt{n\kappa_x\kappa_y})$ SFO upper bound under different smoothness assumptions\footnote{\citet{vladislav2021accelerated} assume each $f_i$ is $L_i$-smooth and $f$ is $L$-smooth, where $L=\frac{1}{n}\sum_{i=1}^nL_i$, while this paper only requires the weaker assumption that $\{f\}_{i=1}^n$ is $L$-average smooth as Definition~\ref{dfn:smooth}. Please see the detailed discussion for these two different settings in Section \ref{sec:acc-scsc}.}. 

In fact, SCSC minimax problems with finite-sum structure can be classified by three types of relationships among $\kappa_x$, $\kappa_y$ and $n$ (recall we have assumed $\kappa_x\geq\kappa_y$)~:
\begin{enumerate}
\item[(a)] $f(\vx, \vy)$ is extremely ill-conditioned w.r.t two variables: $\kappa_x=\Omega(\sqrt{n})$ and $\kappa_y=\Omega(\sqrt{n})$;
\item[(b)] $f(\vx, \vy)$ is only extremely ill-conditioned w.r.t $\vy$: $\kappa_x=\fO(\sqrt{n})$ and $\kappa_y=\Omega(\sqrt{n})$;
\item[(c)] the number of components is extremely large: $\kappa_x=\fO(\sqrt{n})$ and $\kappa_y=\fO(\sqrt{n})$.
\end{enumerate}
\citet{han2021lower} provided a general SFO lower bound for all three cases as follows
\begin{align}\label{complexity:lower-bound}
\begin{cases}
    \displaystyle\Omega\big(\big(n+\sqrt{n\kappa_x\kappa_y}\big)\log(1/\eps)\big), & \text{for~} \kappa_x=\Omega(\sqrt{n})  \text{~and~} \kappa_y=\Omega(\sqrt{n}); \\[0.1cm]
    \displaystyle\Omega\big(\big(n+n^{3/4}\sqrt{\kappa_y}\big)\log(1/\eps)\big), & \text{for~} \kappa_x=\fO(\sqrt{n})  \text{~and~} \kappa_y=\Omega(\sqrt{n}); \\[0.1cm]
    \Omega(n), & \text{for~} \kappa_x=\fO(\sqrt{n})  \text{~and~} \kappa_y=\fO(\sqrt{n}).
\end{cases}
\end{align}
We can observe that the complexity of SVRG/SAGA~\cite{palaniappan2016stochastic} is near optimal in the case of (c). \citet{vladislav2021accelerated}'s algorithm has the upper bound of the form $\fO\left((n+\sqrt{n\kappa_x\kappa_y})\log(1/\eps))\right)$, but it requires the stronger assumption that each component $f_i$ is $L$-smooth. We should point out that Case (b) is also an important setting, which is useful to establish the efficient algorithms for ill-conditioned strongly-convex-nonconcave (or nonconvex-strongly-concave) minimax optimization~\cite{lin2020near,lin2019gradient,luo2020stochastic}.

\begin{table*}[t]
\centering
\renewcommand{\arraystretch}{1.6}
{\small\begin{tabular}{|c|c|c|c|c|}
\hline
Algorithm & SFO Complexity & \!\!Loop\!\! & Reference  \\ \hline
Extragradient & $\fO\big(n\kappa_{\max}\log\left(\frac{1}{\eps}\right)\big)$ & 1 & \citet{gidel2019variational,korpelevich1977extragradient}   \\\hline
\!\!SVRG/SAGA/SVRE\!\! & $\fO\left(\left(n+\kappa_{\max}^2\right)\log\left(\frac{1}{\eps}\right)\right)$ & 1 & \!\!\!\citet{palaniappan2016stochastic,chavdarova2019reducing}\!\!\!  \\ \hline
A-SVRG/A-SAGA & $\tilde\fO\left(\left(n+\sqrt{n}\kappa_{\max}\right)\log\left(\frac{1}{\eps}\right)\right)$ & 2 & \citet{palaniappan2016stochastic}  \\ \hline
L-SVRE & $\fO\left(\left(n+\sqrt{n}\kappa_{\max}\right)\log\left(\frac{1}{\eps}\right)\right)$ & 1 & \citet{alacaoglu2021stochastic} + Theorem \ref{thm:svre} \\ \hline
MINIMAX-APPA & $\tilde\fO\left(n\sqrt{\kappa_x\kappa_y}\log^3\left(\frac{1}{\eps}\right)\right)$ & 3 & \citet{lin2020near}   \\\hline
PBR & $\tilde\fO\left(n\sqrt{\kappa_x\kappa_y}\log\left(\frac{1}{\eps}\right)\right)$ & 3 & \citet{wang2020improved}   \\\hline
 AL-SVRE & \!\!$\tilde\fO\big(\sqrt{n(\sqrt{n}+\kappa_x)(\sqrt{n}+\kappa_y)}\log\big(\frac{1}{\eps}\big)\big)$\!\! &  2 &  Corollary \ref{cor:csc}   \\ \hline
 Lower Bound & the expression of (\ref{complexity:lower-bound}) & -- &  \citet{han2021lower}   \\ \hline
\end{tabular}}
\caption{Comparison of SFO complexities in the $(\mu_x,\mu_y)$-convex-concave setting.}\label{table:SCSC}
\end{table*}

\begin{table*}[t]
\centering
\renewcommand{\arraystretch}{1.85}
{\small\begin{tabular}{|c|c|c|c|c|}
\hline
Algorithm & SFO Complexity & Loop & Reference  \\ \hline
DIAG & $\tilde\fO\left(n\kappa_y\sqrt{\frac{ L}{\eps}}\log^2\left(\frac{1}{\eps}\right)\right)$ & 3 & \citet{thekumparampil2019efficient}   \\\hline
MINIMAX-APPA & $\tilde\fO\left(n\sqrt{\frac{\kappa_y L}{\eps}}\log^3\left(\frac{1}{\eps}\right)\right)$ & 3 & \citet{lin2020near}   \\\hline
A-SVRG/A-SAGA & $\tilde\fO\Big(\Big(n+(\kappa_y^{3/2}+n^{3/4})\sqrt{\frac{L}{\eps}}\Big)\log(\frac{1}{\eps})\Big)$ & 2 & \citet{yang2020catalyst}   \\\hline
 AL-SVRE &  $\tilde\fO\Big(\Big(n + \sqrt{\frac{nL\kappa_y}{\eps}} + n^{3/4} \sqrt{\kappa_y} + n^{3/4}\sqrt{\frac{L}{\eps}}\Big)\log(\frac{1}{\eps}\big)\Big)$ &  2 &  Corollary \ref{cor:cc}   \\ \hline
 Lower Bound & $\Omega\Big(n + \sqrt{\frac{n L \kappa_y}{\eps}} + n^{3/4} \sqrt{\kappa_y} + n^{3/4}\sqrt{\frac{L}{\eps}}\Big)$ & -- &  \citet{han2021lower}   \\ \hline
\end{tabular}}
\caption{Comparison of SFO complexities in the $(0,\mu_y)$-convex-concave setting, where the result of A-SVRG/A-SAGA~\cite{yang2020catalyst} requires the assumption of $\eps<\mu_y$.}\label{table:SCS}
\end{table*}

In this paper, we propose a Catalyst-type algorithm that we call accelerated loopless stochastic variance reduced extragradient (AL-SVRE), whose iteration inexactly solves more well-conditioned minimax problems.
We use loopless stochastic variance reduced extragradient (L-SVRE)~\cite{alacaoglu2021stochastic} as the sub-problem solver and revise \citet{alacaoglu2021stochastic}'s analysis to show L-SVRE solves balanced SCSC minimax problem in optimal SFO complexity. 
Combining an appropriate choice of parameters, AL-SVRE could find an approximate saddle point of our main problem (\ref{prob:main}) with SFO complexity nearly matching \citet{han2021lower}'s lower bound (\ref{complexity:lower-bound}).
Additionally, AL-SVRE only applies one times Catalyst acceleration on AL-SVRE, which leads to the algorithm has two-loops of iterations in total and it is easily implemented. The empirical studies on AUC maximization~\cite{hanley1982meaning,ying2016stochastic,shen2018towards} and  wireless communication~\cite{garnaev2009eavesdropping,boyd2004convex,yang2020catalyst} problems  show that AL-SVRE performs better than baselines.
In contrast, previous Catalyst-type methods~\cite{lin2020near,wang2020improved,vladislav2021accelerated} for such an unbalanced problem need twice Catalyst acceleration (three-loops of iterations), making these algorithms almost impractical. 

We can also apply AL-SVRE to solve more general finite-sum convex-concave minimax problems. If we suppose $f(\vx,\vy)$ is $\mu_y$-strongly-concave in $\vy$ and allow it could be non-strongly-convex in $\vx$, AL-SVRE could find an $\eps$-saddle point of (\ref{prob:main}) in 
\begin{align}\label{complexity:upper-bound-NSS}
\tilde\fO\left(\left(n+D_x \sqrt{\frac{n L \kappa_y}{\eps}}+n^{3/4}\sqrt{\kappa_y}+n^{3/4}D_x \sqrt{\frac{L}{\eps}}\right)\log\left(\frac{1}{\eps}\right)\right)
\end{align}
SFO complexity, where $D_x$ is the diameter of $\fX$. The upper bound (\ref{complexity:upper-bound-NSS}) nearly matches the lower bounds with respect to $\eps$, $L$, $\mu_y$ and $n$ simultaneously, and we do not require any additional assumption on $L$, $\mu_y$, $n$ and $\eps$, while the state-of-the-art method~\cite{yang2020catalyst} implicitly requires that $\eps<\mu_y$.

More general, if we only suppose the objective function is convex in $\vx$ and concave in $\vy$, the algorithm could find an $\eps$-saddle point in
\begin{align}\label{complexity:upper-bound-NSNS}
\tilde\fO\left(\left(n+\frac{\sqrt{n} L D_x D_y}{\eps}+n^{3/4} (D_x+D_y) \sqrt{\frac{L}{\eps}}\right)\log\left(\frac{1}{\eps}\right)\right)
\end{align}
SFO complexity, where $D_x$ and $D_y$ are diameters of $\fX$ and $\fY$ respectively. Note that the upper bound (\ref{complexity:upper-bound-NSNS}) nearly matches the lower bounds~\cite{han2021lower} in such a setting. 
Compared with the the best known stochastic SFO algorithm L-SVRE~\cite{alacaoglu2021stochastic}, our result additionally trades off the difference between $D_x$ and $D_y$.

We present comparisons between our results and existing results in Table~\ref{table:SCSC} for strongly-convex-strongly-concave setting, in Table~\ref{table:SCS} for strongly-convex-concave setting, and in Table~\ref{table:CC} for convex-concave settings.
To the best of our knowledge, the proposed AL-SVRE is the first algorithm which attains near optimal SFO complexities for all the above settings.

\section{Notation and Preliminaries}

In this section, we present the notation and some definitions used in this paper. 

\begin{table*}[t]
\centering
\renewcommand{\arraystretch}{1.8}
{\small\begin{tabular}{|c|c|c|c|c|}
\hline
Algorithm & SFO Complexity & Loop & Reference  \\ \hline
Extragradient & $\fO\big(\max\{D_x^2,D_y^2\}\frac{nL}{\eps}\big)$ & 1 & \citet{gidel2019variational,korpelevich1977extragradient}   \\\hline
MINIMAX-APPA & $\tilde\fO\big(\frac{nLD_xD_y}{\eps}\log^3\left(\frac{1}{\eps}\right)\big)$ & 3 & \citet{lin2020near}   \\\hline
L-SVRE & $\fO\big(\max\{D_x^2,D_y^2\}(n+\frac{\sqrt{n}L}{\eps})\big)$ & 1 & \citet{alacaoglu2021stochastic}   \\\hline
AL-SVRE &  $\tilde\fO\Big(\Big(n + \frac{\sqrt{n} L D_x D_y}{\eps} + n^{3/4} (D_x + D_y) \sqrt{\frac{L}{\eps}}\Big)\log\Big(\frac{1}{\eps}\Big)\Big)$ &  2 &  Corollary \ref{cor:ASVRE}   \\ \hline
 Lower Bound & $\Omega \Big(n + \frac{\sqrt{n} L D_x D_y}{\eps} + n^{3/4} (D_x + D_y) \sqrt{\frac{L}{\eps}}\Big)$ & -- &  \citet{han2021lower}   \\ \hline
\end{tabular}}
\caption{Comparison of SFO complexities in convex-concave setting.}\label{table:CC}
\end{table*}

\begin{dfn}\label{dfn:smooth}
For any differentiable function $\psi\colon \fZ \to \BR$, we say $\psi$ is $L$-smooth for some $L>0$ if for any $\vz, \vz'\in\fZ$, it holds that
$\|\nabla \psi(\vz) - \nabla \psi(\vz')\|_2 \leq L \norm{\vz - \vz'}$.
\end{dfn}

\begin{dfn}
Suppose there are $n$ differentiable functions $\{\psi_i\colon \fZ \to \BR\}_{i=1}^n$. We say $\{\psi_i\}_{i=1}^n$ is $L$-average smooth for some $L>0$ if for any $\vz, \vz' \in \fZ$, it holds that
\begin{align*}
    \frac{1}{n} \sum_{i=1}^n \norm{\nabla \psi_i(\vz) - \nabla \psi_i(\vz')}^2 \le L^2\norm{\vz - \vz'}^2.
\end{align*}
\end{dfn}

\begin{dfn}
For a differentiable function $\psi\colon \fZ \rightarrow \BR$, we say $\psi$ is convex if for any $\vz, \vz' \in \fZ$, it holds that $\psi(\vz') \ge \psi(\vz) +  \inner{\nabla\psi(\vz)}{\vz' - \vz}$.
We say $\psi$ is $\mu$-strongly-convex for some $\mu>0$ if $\psi(\cdot)-\frac{\mu}{2}\norm{\cdot}^2$ is convex.
We also say $\psi$ is concave ($\mu$-strongly-concave) if $-\psi$ is convex ($\mu$-strongly-convex).  
\end{dfn}


\begin{dfn}\label{dfn:scsc}
For any function $f\colon \fX\times\fY\rightarrow\BR$ and $\mu_x, \mu_y \geq0$, we say $f$ is $(\mu_x,\mu_y)$-convex-concave if for any $\vx\in\fX$ and $\vy\in\fY$, it holds that $f(\vx,\cdot)$ is $\mu_y$-strongly-concave and $f(\cdot,\vy)$ is $\mu_x$-strongly-convex.
\end{dfn}

In Definition \ref{dfn:scsc}, we allow both $\mu_x$ and $\mu_y$ to be zero. The notation $(0,0)$-convex-concave means the function is general convex-concave, and  $(0, \mu_y)$-convex-concave means it is $\mu_y$-strongly-concave in $\vy$ but possibly non-strongly-convex in $\vx$. Similarly, we use $(\mu_x, 0)$-convex-concave to present the function is $\mu_x$-strongly-convex in $\vx$ but possibly non-strongly-concave in $\vy$.

We are interested in finding an approximate saddle point which is defined as follows.
\begin{dfn}
For the minimax optimization problem (\ref{prob:main}),  $(\hat \vx, \hat \vy)\in\fX\times\fY$ is said to be an $\eps$-saddle point if
\begin{align*}
\max_{\vy\in\fY} f(\hat \vx, \vy) - \min_{\vx\in\fX} f(\vx, \hat \vy) \leq \eps.
\end{align*}
\end{dfn}

We introduce the (stochastic) gradient operator for the ease of presentation.
\begin{dfn}
For problem (\ref{prob:main}), we define the gradient operator and stochastic gradient operator as 
\begin{align*}
\vg(\vz) = \begin{bmatrix}
\nabla_\vx f(\vz) \\ -\nabla_\vy f(\vz)
\end{bmatrix}
\quad \text{and} \quad
\vg_i(\vz) = \begin{bmatrix}
\nabla_\vx f_i(\vz) \\ -\nabla_\vy f_i(\vz)
\end{bmatrix}
\end{align*}
where $\vz=(\vx,\vy)\in\BR^{d_x}\times\BR^{d_y}$.
\end{dfn}

We also conduct the projection operator to address the constraints in the problem.
\begin{dfn}
We define the projection of $\vz$ onto the convex and compact set $\fC$ as
\begin{align*}
    \fP_\fC(\vz)=\argmin_{\vu\in\fC}\norm{\vu-\vz}^2.
\end{align*}
\end{dfn}

This paper focuses on solving finite-sum minimax problems by SFO algorithms and we give the formal definition as follows.

\begin{dfn}\label{dfn:SFO}
Consider a stochastic optimization algorithm $\fA$ to solve Problem (\ref{prob:main}). Suppose an initial point $(\vx^{(0)}, \vy^{(0)})$ is given, and let $(\vx^{(t)}, \vy^{(t)})$ denote the point obtained by $\fA$ at time-step $t$. The algorithm is said to be an SFO algorithm if for any $t > 0$, we have
\begin{align*}
\vx^{(t)} = \fP_\fX\big(\tilde\vx^{(t)}\big) \text{~~~and~~~}
    \vy^{(t)} = \fP_\fY\big(\tilde\vy^{(t)}\big),
\end{align*}
where
\begin{align*}
\tilde\vx^{(t)} \in & \spn\big\{\vx^{(0)}, \vx^{(1)} \dots, \vx^{(t-1)}, \nabla_{\vx} f_{i_{t}}\big(\vx^{(0)}, \vy^{(0)}\big),
\dots, \nabla_{\vx} f_{i_{t}}\big(\vx^{(t-1)}, \vy^{(t-1)}\big)\big\}, \\
\tilde\vy^{(t)} \in & \spn\big\{\vy^{(0)}, \vy^{(1)} \dots, \vy^{(t-1)}, \nabla_{\vy} f_{i_{t}}\big(\vx^{(0)}, \vy^{(0)}\big), 
\dots, \nabla_{\vy} f_{i_{t}}\big(\vx^{(t-1)}, \vy^{(t-1)}\big)\big\},
\end{align*}
and $i_t$ is drawn from $\{1,\dots,n\}$.
\end{dfn}

\section{Accelerating Unbalanced Convex-Concave Optimization}

In this section, we first revisit loopless stochastic variance reduced extragradient (L-SVRE)~\cite{alacaoglu2021stochastic}, and show it can be used to solve balanced strongly-convex-strongly-concave (SCSC) minimax with optimal SFO complexity. 
Then we propose a Catalyst-type scheme to accelerate L-SVRE, deriving our new algorithm \emph{Accelerated Loopless Stochastic Variance Reduced Extragradient} (AL-SVRE) whose SFO upper bound is near optimal for different types of unbalanced concave-concave minimax problems.

\begin{algorithm}[t]
\caption{$\text{L-SVRE}~\left(\{f_i(\vx,\vy)\}_{i=1}^n, (\vx_0, \vy_0), \tau, p, T\right)$} \label{algo:svre}
\begin{algorithmic}[1]
    \STATE \textbf{Initialize:} $\alpha = 1 - p$, $\vw_0 = \vz_0$. \\[0.15cm]
    \STATE \textbf{for} $k = 0, 1, \dots, T-1$ \textbf{do}\\[0.15cm]
    \STATE\quad $\bar{\vz}_k = \alpha \vz_k + (1 - \alpha) \vw_k$ \\[0.15cm]
    \STATE\quad $\vz_{k+1/2} = \fP_{\fZ}(\bar{\vz}_k - \tau \vg(\vw_k))$ \\[0.15cm]
    \STATE\quad Draw an index $i \in [n]$ uniformly at random. \\[0.15cm]
    \STATE\quad $\vz_{k+1} = \fP_{\fZ}(\bar{\vz}_k - \tau [\vg(\vw_k) + \vg_i(\vz_{k+1/2}) - \vg_i(\vw_k)])$ \\[0.15cm]
    \STATE\quad $\vw_{k+1} = \begin{cases} \vz_{k+1}, ~&\text{ with probability } p \\
        \vw_k, ~&\text{ with probability } 1 - p \end{cases}$ \\[0.15cm]
    \STATE\textbf{end for} \\[0.15cm]
    \STATE \textbf{Output:} $(\vx_T, \vy_T)$.
\end{algorithmic}
\end{algorithm}

\subsection{The Optimal SFO Algorithm for Balanced SCSC Minimax}

We start our discussion from the optimality of balanced SCSC problems such that the objective function $f(\vx,\vy)$ is $(\mu,\mu)$-convex-concave with some $\mu>0$.  
Theorem \ref{thm:balance-lower-bound} gives an SFO lower bound $\Omega((n+\sqrt{n}\kappa)\log(1/\eps))$ for finding an approximate saddle point with respect to the square of Euclidean distance to the optimal solution under average smooth assumptions. 
This result is slightly different from previous work which is based on duality gap convergence~\cite{han2021lower} or under assumption that each component $f_i$ is $L$-smooth~\cite{xie2020lower}.

\begin{thm}\label{thm:balance-lower-bound}
For any SFO algorithm $\fA$ and $L, \mu, n, \eps$ such that $L/\mu > 2$ and $\eps < 0.003$, there exist a dimension $d = \fO(n + \sqrt{n} L/\mu \log(1/\eps))$ and functions $\{f_i(\vx, \vy)\}_{i=1}^n\colon \BR^{d} \times \BR^d \to \BR$ which satisfy that $\{f_i\}_{i=1}^n$ is $L$-average smooth, $f$ is $(\mu,\mu)$-convex-concave. In order to find an approximate saddle point $(\hat\vx,\hat\vy)$ such that 
\begin{align*}
    \BE\left[\norm{\hat\vx-\vx^*}^2+\norm{\hat\vy-\vy^*}^2\right] \leq \eps,
\end{align*}
algorithm $\fA$ needs at least $\Omega(n {+} \sqrt{n}L/\mu \log(1/\eps))$ SFO calls.
\end{thm}

Recently, \citet{alacaoglu2021stochastic} proposed the L-SVRE algorithm, which combines the idea of loopless SVRG~\cite{kovalev2020don} and extragradient~\cite{korpelevich1977extragradient}.
The original motivation of L-SVRE is to solve variational inequalities. We revise its analysis to adapt the standard finite-sum minimax problem (\ref{prob:main}) under the $(\mu,\mu)$-convex-concave assumption. We present the details of L-SVRE in Algorithm~\ref{algo:svre} and show its convergence behavior in Theorem \ref{thm:svre}. 

\begin{algorithm*}[t]
    \caption{$\text{AL-SVRE}\left(\{f_i(\vx,\vy)\}_{i=1}^n, (\vx_0, \vy_0), \beta, q, K, p, \{\tau_k\}_{k=1}^K, \{T_k\}_{k=1}^K \right)$}\label{algo:asvre}
    \begin{algorithmic}[1]
        \STATE \textbf{Initialize:} $\gamma = \frac{1 - \sqrt{q}}{1 + \sqrt{q}}$ and $\vu_0 = \vx_0$. \\[0.15cm]
        \STATE \textbf{for} $k = 1, \cdots, K$ \textbf{do}\\[0.15cm]
        \STATE\quad $(\tilde\vx_k, \tilde\vy_k) = \text{L-SVRE}~\big(\big\{f_i(\vx,\vy)+\frac{\beta}{2}\norm{\vx-\vu_{k-1}}^2\big\}_{i=1}^n, \vx_{k-1}, \vy_{k-1}, \tau_k, p, T_k\big)$ \\[0.15cm]
        \STATE\quad $\vx_k = \fP_{\fX} \left(\tilde\vx_k - \tau_k\nabla_{\vx} F_k\big(\tilde\vx_k, \tilde\vy_k\big)\right)$ \\[0.15cm]
        \STATE\quad $\vy_k = \fP_{\fY} \left(\tilde\vy_k + \tau_k\nabla_{\vy} F_k\big(\tilde\vx_k, \tilde\vy_k\big)\right)$  \\[0.15cm]
        \STATE\quad $\vu_k=\vx_k + \gamma(\vx_k-\vx_{k-1})$ \\[0.15cm]
        \STATE\textbf{end for} \\[0.15cm]
        \STATE \textbf{Output:} $(\vx_K, \vy_K)$.
    \end{algorithmic}
\end{algorithm*}
\begin{thm}\label{thm:svre}
Assume that $\mu_x = \mu_y = \mu$ and $\{f_i\}_{i=1}^n$ is $L$-average smooth. Then Algorithm \ref{algo:svre} with probability parameter $p = 1/(2 n)$ and stepsize $\tau = 1/(4\sqrt{n} L)$ satisfies
\begin{align}\label{ieq:convege-SVRE}
\E\norm{\vz_{k} - \vz^*}^2 
\le  4 \norm{\vz_0 - \vz^*}^2 \left(1 - \frac{1}{4(n + 2 \sqrt{n} L / \mu)}\right)^{k}\!.
\end{align}
\end{thm}

The result of inequality (\ref{ieq:convege-SVRE}) means that L-SVRE could find $(\hat\vx,\hat\vy)$ satisfying $\BE[\norm{\hat\vx-\vx^*}^2+\norm{\hat\vy-\vy^*}^2]\leq\eps$ with $\fO((n {+} \sqrt{n}L/\mu)\log(1/\eps))$ iterations. Since we select $p=1/2n$ in the theorem, each iteration requires $\fO(1)$ SFO calls in expectation. Hence, we have proved that the upper bound complexity of L-SVRE matches the lower bound shown in Theorem~\ref{thm:balance-lower-bound}.

\subsection{Acceleration for Unbalanced SCSC Minimax}\label{sec:acc-scsc}

Note that the convergence result of L-SVRE is not perfect when the objective function is unbalanced. For example, we consider problem (\ref{prob:main}) in the case of $\kappa_x=\kappa>\sqrt{n}>\kappa_y=\fO(1)$. Then Theorem \ref{thm:svre} cannot leverage the well-conditioned assumption on $\vy$ and leads to SFO upper bound $\fO\left(\sqrt{n}\kappa\log(1/\eps))\right)$, which is 
worse than complexity $\tilde\fO(n\sqrt{\kappa}\log^3(1/\eps))$ achieved by deterministic algorithms~\cite{lin2020near}.

Our key intuition to establish a better SFO algorithm for unbalanced SCSC minimax problems is taking the advantage of the optimality of L-SVRE in the balanced case. We present the new method AL-SVRE in Algorithm~\ref{algo:asvre}, whose iteration applies L-SVRE to solve the following sub-problem
\begin{align}\label{prob::sub}
\min_{\vx\in\fX} \max_{\vy\in\fY} F_k(\vx, \vy) \triangleq f(\vx, \vy) + \frac{\beta}{2}\norm{\vx-\vu_{k-1}}^2.
\end{align}
Since Theorem \ref{thm:svre} only provides the convergence rate for L-SVRE by distance, we introduce additional projection gradient iterations (Lines 4-5 of Algorithm \ref{algo:asvre}) on the output of the sub-problem solver.
These steps help us control the accuracy with respect to primal and dual functions, which is helpful to establish the convergence results with respect to the duality gap in our main result Theorem~\ref{thm:asvre}.

\begin{lem}\label{lem:SVRE-prime-dual}
Suppose the function $f(\vx, \vy)\colon \fX \times \fY \to \BR$ is $(\mu_x, \mu_y)$-convex-concave and $L$-smooth. Denote the saddle point of $\min_{\vx\in\fX}\max_{\vy\in\fY}f(\vx,\vy)$ by $(\vx^*, \vy^*)$ and let the condition numbers of $f$ be $\kappa_x=L/\mu_x$ and $\kappa_y=L/\mu_y$. Assume that the point $(\hat\vx, \hat\vy)$ satisfies
\begin{align*}
    \norm{\hat\vx - \vx^*}^2 + \norm{\hat\vy - \vy^*}^2 \le \eps.
    \end{align*}
We introduce
\begin{align*}
\tilde\vx = \fP_{\fX} \left( \hat\vx - \eta\nabla_{\vx} f(\hat\vx, \hat\vy) \right)
~~\text{and}~~
\tilde\vy = \fP_{\fY} \left( \hat\vy + \eta\nabla_{\vy} f(\hat\vx, \hat\vy) \right),
\end{align*}
then it holds that 
\begin{align*}
\max_{\vy \in \fY} f(\tilde\vx, \vy) - f(\vx^*, \vy^*)
\le \left( \sqrt{2}(1 + \eta L) + 2(1 + \eta L)^2 + 2 \right) \kappa_y L \eps + \frac{\eps}{2\eta}
\end{align*}
and
\begin{align*}
 f(\vx^*, \vy^*) - \min_{\vx \in \fX} f(\vx, \tilde\vy) 
\le \left( \sqrt{2}(1 + \eta L) + 2(1 + \eta L)^2 + 2 \right) \kappa_x L \eps + \frac{\eps}{2\eta}.
\end{align*}
\end{lem}

The following lemma shows that we can upper bound the distance from given point to the saddle point by its primal-dual gap.
\begin{lem}\label{lem:dist-prime-dual}
Suppose the function $f(\vx, \vy): \fX \times \fY \to \BR$ is $(\mu_x, \mu_y)$-convex-concave and $L$-smooth. Denote the saddle point of $f$ by $(\vx^*, \vy^*)$. Then for any $\hat\vx \in \fX, \hat\vy \in \fY$, we have
\begin{align*}
\mu_x \norm{\hat\vx - \vx^*}^2 + \mu_y \norm{\hat\vy - \vy^*}^2
\le 2\left(\max_{\vy \in \fY} f(\hat\vx, \vy) - \min_{\vx \in \fX} f(\vx, \hat\vy)\right).
\end{align*}
\end{lem}

Then introduces four auxiliary quantities and provide their upper bounds, which are useful to the analyze the convergence of AL-SVRE.

\begin{lem}\label{lem:aux}
We use the notation of Algorithm~\ref{algo:asvre} and denote $(\vx_k^*, \vy_k^*)$ is the saddle point of $F_k$ and
\begin{align*}
        \Delta_f = \max_{\vy \in \fY} f(\vx_0, \vy) - \min_{\vx \in \fX} f(\vx, \vy_0).
\end{align*}
Then, for each $k\geq 1$, we have
\begin{align}
\label{eq:induction-2} &
\E \left[\norm{\vx_{k} - \vx_k^*}^2 + \norm{\vy_{k} - \vy_k^*}^2 \right] \le \eps_k, \\
\label{eq:induction-3} &
\E\left[ \max_{\vy \in \fY} F_k(\vx_k, \vy) - F_k(\vx_k^*, \vy_k^*)\right] \le \frac{2\Delta_f}{9} (1 - \rho)^k, \\
\label{eq:induction-4} &
\E\norm{\vx_k - \vx^*}^2 \le \frac{2\delta_k}{\mu_x}, \\
\label{eq:induction-6} &
\E\left[ \norm{\vx^*_k - \vx^*_{k+1}}^2\!+\!\norm{\vy^*_k - \vy^*_{k+1}}^2 \right] \le \frac{72\beta \delta_{k-2}}{\mu_x \min\{\mu_x, \mu_y\}},
\end{align}
where 
\begin{align*}
    \eps_k \triangleq \frac{2 \mu_y \Delta_f (1 - \rho)^k}{3(L+\beta)(7(L+\beta)+2\sqrt{n}\mu_y)}
\end{align*}
and
\begin{align*}
    \delta_k \triangleq \frac{8\Delta_f (1 - \rho)^{k+1}}{(\sqrt{q} - \rho)^2}.
\end{align*}
\end{lem}
    
Based on above lemmas, we obtain the main result of this paper as follows.
\begin{thm}\label{thm:asvre}
Running Algorithm \ref{algo:asvre} with $q = \mu_x/(\mu_x + \beta)$, $\beta \ge 0$, $p=1/2n$, $\tau_k=1/(4\sqrt{n}(L+\beta))$ and
\begin{align*}
T_k = \Bigg\lceil4\left(n + \frac{2 \sqrt{n}(L+\beta)}{\min\{\mu_x + \beta, \mu_y\}}\right)
\log\left(\!12\!\left(\frac{2}{1-\rho}\!+\!\frac{1728 \beta(L + \beta)(7(L+\beta) + 2\sqrt{n}\mu_y))}{\mu_x\mu_y \min\{\mu_x, \mu_y\} (1 - \rho)^2(\sqrt{q} - \rho)^2}\right) \right)\Bigg\rceil,
\end{align*}
where $\rho < \sqrt{q}$,
then  we have 
\begin{align*}
    \E \left[\max_{\vy \in \fY} f(\vx_k, \vy) - \min_{\vx\in\fX} f(\vx, \vy_k)\right] 
\le \frac{916\Delta_f\left( \kappa_x L + \sqrt{n}(L + \beta) \right)\kappa_y^2}{\mu_x(\sqrt{q} - \rho)^2} (1 - \rho)^{k},
\end{align*}
where $\Delta_f = \max_{\vy \in \fY} f(\vx_0, \vy) - \min_{\vx \in \fX} f(\vx, \vy_0)$.
\end{thm}

Theorem \ref{thm:asvre} shows  the $\eps$-saddle point can be obtained by calling $K = \tilde\fO\big(\sqrt{(\mu_x {+} \beta)/\mu_x}\log(1/\eps)\big)$ times AL-SVRE to solve the sub-problem. By minimizing the product of dominant terms (ignore logarithmic terms) $\sqrt{(\mu_x+\beta)/\mu_x}$ and $\left(n + 2 \sqrt{n}(L+\beta)/\min\{\mu_x + \beta, \mu_y\}\right)$ in $K$ and $T_k$ respectively, we decide the choice of $\beta$ as follows (since we have suppose $\mu_x\leq\mu_y$)
\begin{align}\label{choice:beta}
\beta = \begin{cases}
    \mu_y - \mu_x, ~~&\text{if } \kappa_y \ge \sqrt{n}, \\
        L/\sqrt{n} - \mu_x, ~~&\text{if } \kappa_x > \sqrt{n} > \kappa_y, \\
    0, ~~&\text{otherwise,}
\end{cases}
\end{align}

Combing the result of Theorem \ref{thm:asvre} and equation (\ref{choice:beta}), we immediately obtain the following corollary, which implies AL-SVRE has near optimal SFO upper bound for unbalanced SCSC minimax optimization.
\begin{cor}\label{cor:ASVRE}
Suppose Problem (\ref{prob:main}) satisfies $\mu_x < \mu_y$. If running AL-SVRE with the setting of Theorem \ref{thm:asvre} and letting $\beta$ as (\ref{choice:beta}), $\rho = 0.5 \sqrt{q}$  and 
\begin{align*}
    K = \ceil{\frac{2}{\sqrt{q}} \log\left( 10992\sqrt{n}\Delta_f \kappa_y \kappa_x^3\right) },
    \end{align*}
then we have 
\begin{align*}
    \E \Big[\max_{\vy \in \fY} f(\vx_K, \vy) - \min_{\vx\in\fX} f(\vx, \vy_K)\Big] \le \eps 
    \end{align*}
with the number of SFO calls at most
{\small\begin{align*}
\fO\left(\sqrt{n(\sqrt{n} + \kappa_x)(\sqrt{n} + \kappa_y)}\log(n\kappa_x) \log\left(\frac{n\Delta_f \kappa_y \kappa_x}{\eps}\right) \right).
\end{align*}}
\end{cor}

The design of AL-SVRE aims to sufficiently take the advantage of AL-SVRE's optimality for the balanced SCSC minimax problem. Hence, the choice of $\beta$ makes the sub-problem (\ref{prob::sub})  more balanced than the main problem (\ref{prob:main}). In contrast, previous catalyst-type methods~\cite{lin2020near,wang2020improved,vladislav2021accelerated} select a larger $\beta$ to guarantee the sub-problem be well-conditioned in one of the variables. However, such strategy leads to that the sub-problem is still so unbalanced and we need two-loops of iterations to solve it. As a result, the implementation of the whole procedure is more complicated and almost cannot be used in practical.

\paragraph{Discussion}

Note that \citet{vladislav2021accelerated}'s work supposes each $f_i(\vx,\vy)$ is $L_i$-smooth, that is
\begin{align}\label{asm:each-smooth}
    \norm{\nabla f_i(\vz_1) - \nabla f_i(\vz_2)} \leq L_i\norm{\vz_1-\vz_2}
\end{align}
where $\vz_1=(\vx_1,\vy_1)$ and $\vz_2=(\vx_2,\vy_2)$, which leads to the objective function $f(\vz)=\frac{1}{n}\sum_{i=1}^n f_i(\vz)$ is $L_G$-smooth with $L_G=\frac{1}{n}\sum_{i=1}^n L_i$.
Furthermore, we can verify that $\{f_i\}_{i=1}^n$ is $L_A$-average smooth, where $L_A = \sqrt{\frac{1}{n} \sum_{i=1}^n L_i^2}$ (c.f. Lemma \ref{lem:smooth-average-smooth}).

Consider a simple class of functions $\{g_i(\vz)\}_{i=1}^n$ such that $g_i(\vz) = \frac{\sqrt{n} L}{2} (\ve_i^{\top} \vz)^2$, where $\{\ve_i\}_{i=1}^n$ is the standard basis of $\BR^n$.  
It is easily to check that $g_i$ is $\sqrt{n} L$-smooth and
\begin{align*}
    \frac{1}{n} \sum_{i=1}^n \norm{\nabla g_i(\vz_1) - \nabla g_i(\vz_2)}^2 
    = \frac{1}{n} \sum_{i=1}^n (n L^2) \norm{\ve_i \ve_i^{\top} (\vz_1 - \vz_2)}^2
    = L^2 \sum_{i=1}^n (z_{1, i} - z_{2, i})^2 = L^2 \norm{\vz_1 - \vz_2}^2,
\end{align*}
which implies that $\{g_i\}_{i=1}^n$ is $L$-average smooth.


Using notations of~\citet{vladislav2021accelerated} on $\{g_i\}_{i=1}^n$, we have $L_i=\sqrt{n} L$ and $L_G=\frac{1}{n}\sum_{i=1}^n L_i=\sqrt{n}L$. Hence, their smoothness parameter $L_G$ is larger than the average-smoothness parameter $L$ we used in this paper.

\section{Extensions to Non-SCSC Minimax}\label{sec:extensions}

We can also apply AL-SVRE to solve convex-concave minimax without SCSC assumption.

\subsection{Convex-Strongly-Concave Case} 
If the objective function $f(\vx,\vy)$ is $(0,\mu_y)$-convex-concave and $\fX$ is bounded with diameter $D_x$, we construct the auxiliary SCSC function to approximate it as follows:
\begin{align}\label{fun:csc}
    f_{\eps, \vx_0} (\vx, \vy) \triangleq f(\vx, \vy) + \frac{\eps}{4 D_x^2} \norm{\vx - \vx_0}^2.
\end{align}
Then the difference of duality gaps between $f_{\eps, \vx_0}$ and $f$ should be small. The following lemma presents this fact formally.
\begin{lem}\label{lem:csc}
Suppose $f(\vx, \vy)$ is $(0, \mu_y)$-convex-concave, $\fX$ is bounded with diameter $D_x$ and $\vx_0\in\fX$. Consider the function
$f_{\eps, \vx_0} (\vx, \vy)$ defined as (\ref{fun:csc}).
Then for any $(\hat\vx, \hat\vy) \in \fX \times \fY$, we have
\begin{align*}
\max_{\vy \in \fY} f (\hat{\vx}, \vy) - \min_{\vx \in \fX} f (\vx, \hat{\vy}) 
\le \frac{\eps}{2} + \max_{\vy \in \fY} f_{\eps, \vx_0} (\hat{\vx}, \vy) - \min_{\vx \in \fX} f_{\eps, \vx_0} (\vx, \hat{\vy}).
\end{align*}
\end{lem}

Lemma \ref{lem:csc} means that any $\eps/2$-saddle-point of $f_{\eps, \vx_0}$ is also an $\eps$-saddle-point of $f$. Hence, we can directly run AL-SVRE on $f_{\eps,\vx_0}$ and connect Lemma~\ref{lem:csc} and Theorem~\ref{thm:asvre} to establish the convergence result for $(0,\mu_y)$-convex-concave minimax optimization as follows.

\begin{figure*}[t]
\centering
\begin{tabular}{cccc}
\includegraphics[scale=0.29]{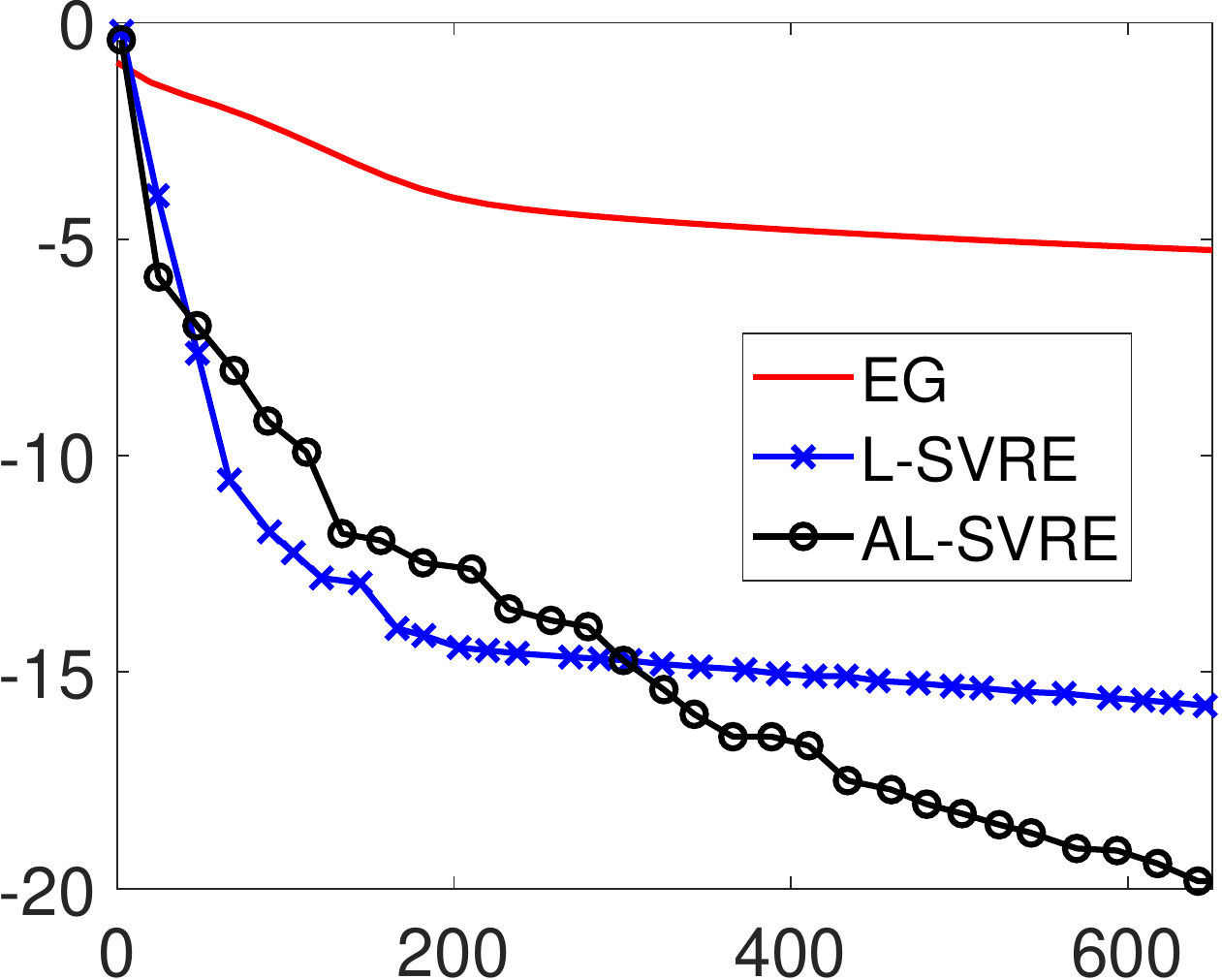}  & 
\includegraphics[scale=0.29]{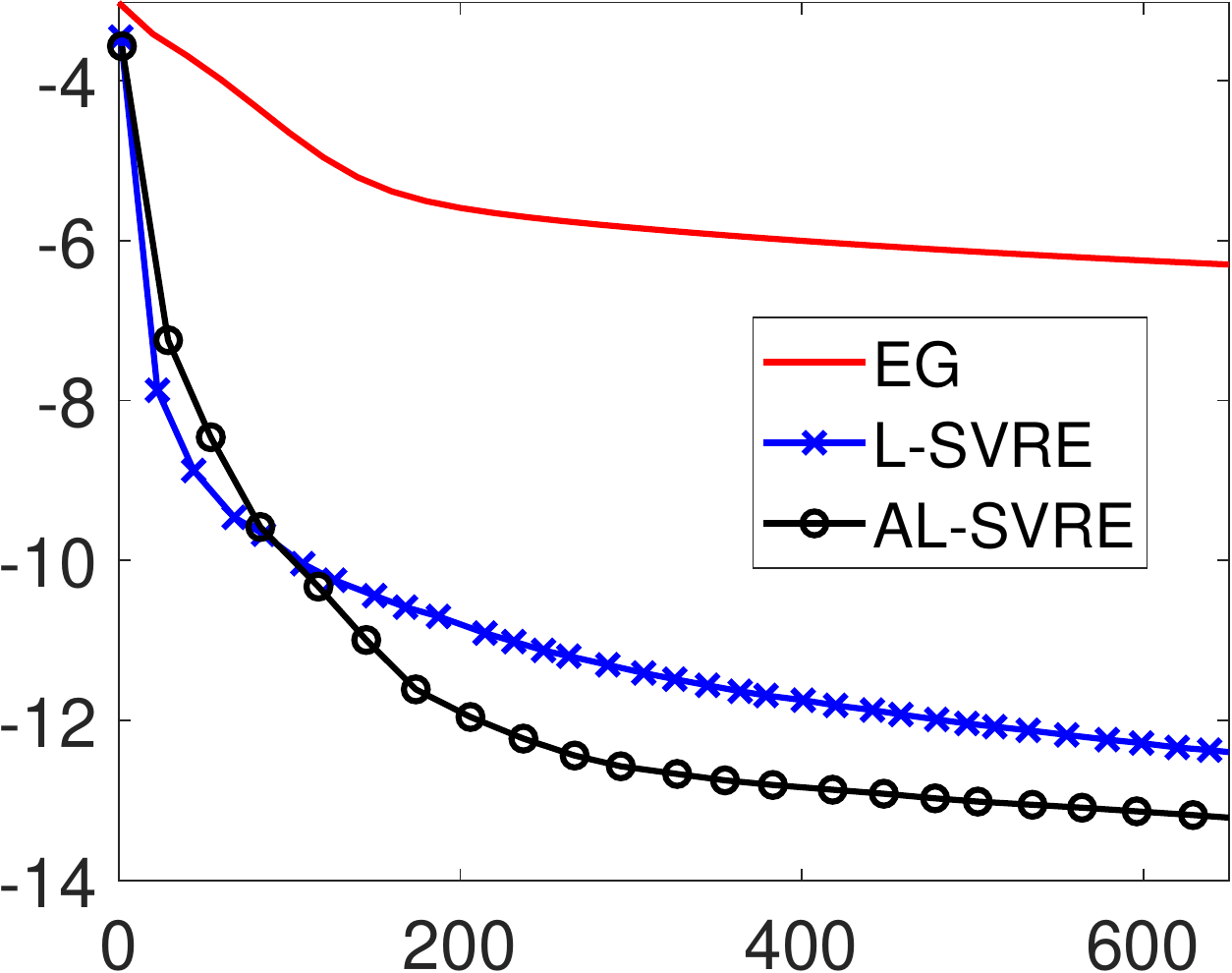} &
\includegraphics[scale=0.29]{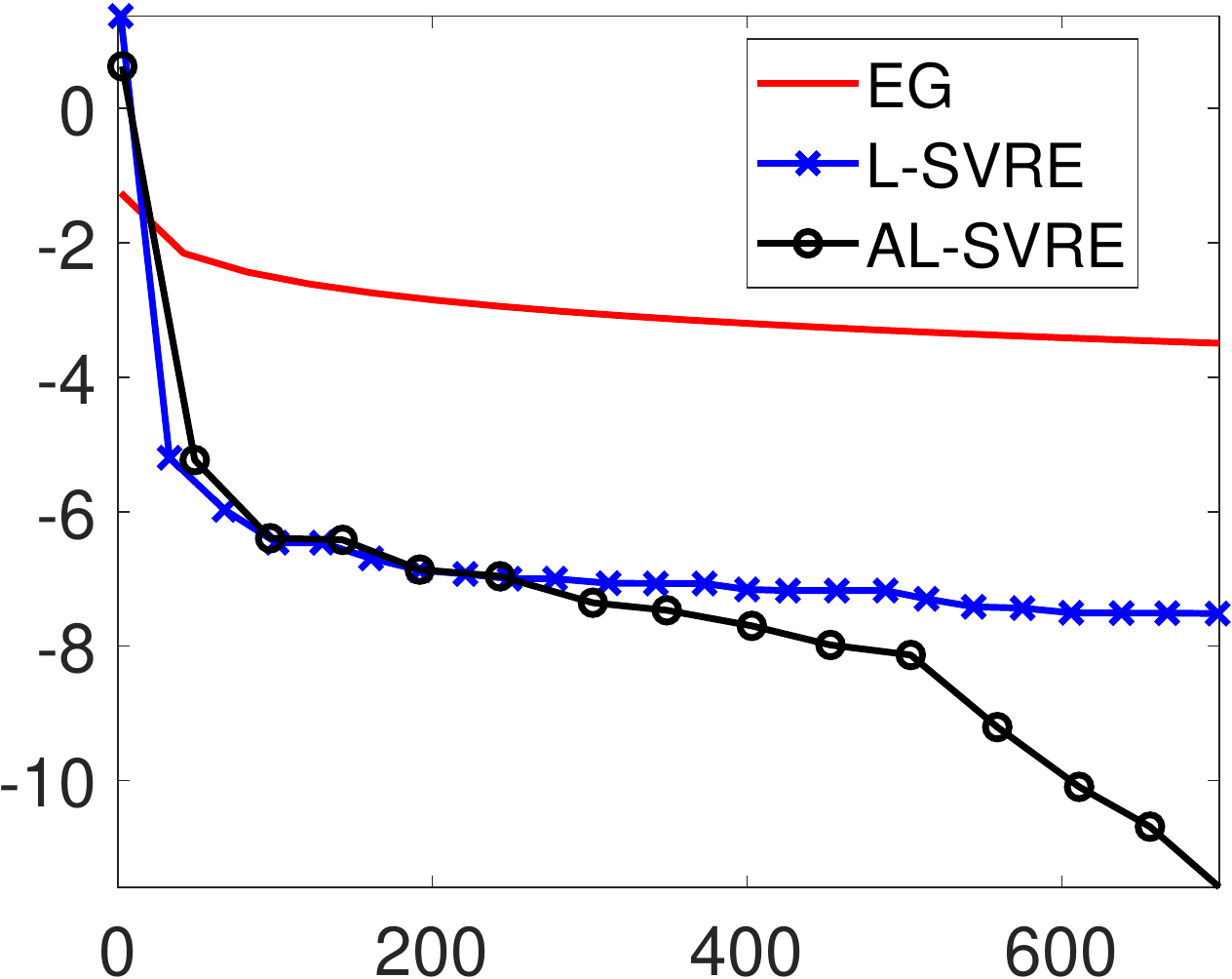} &  \includegraphics[scale=0.29]{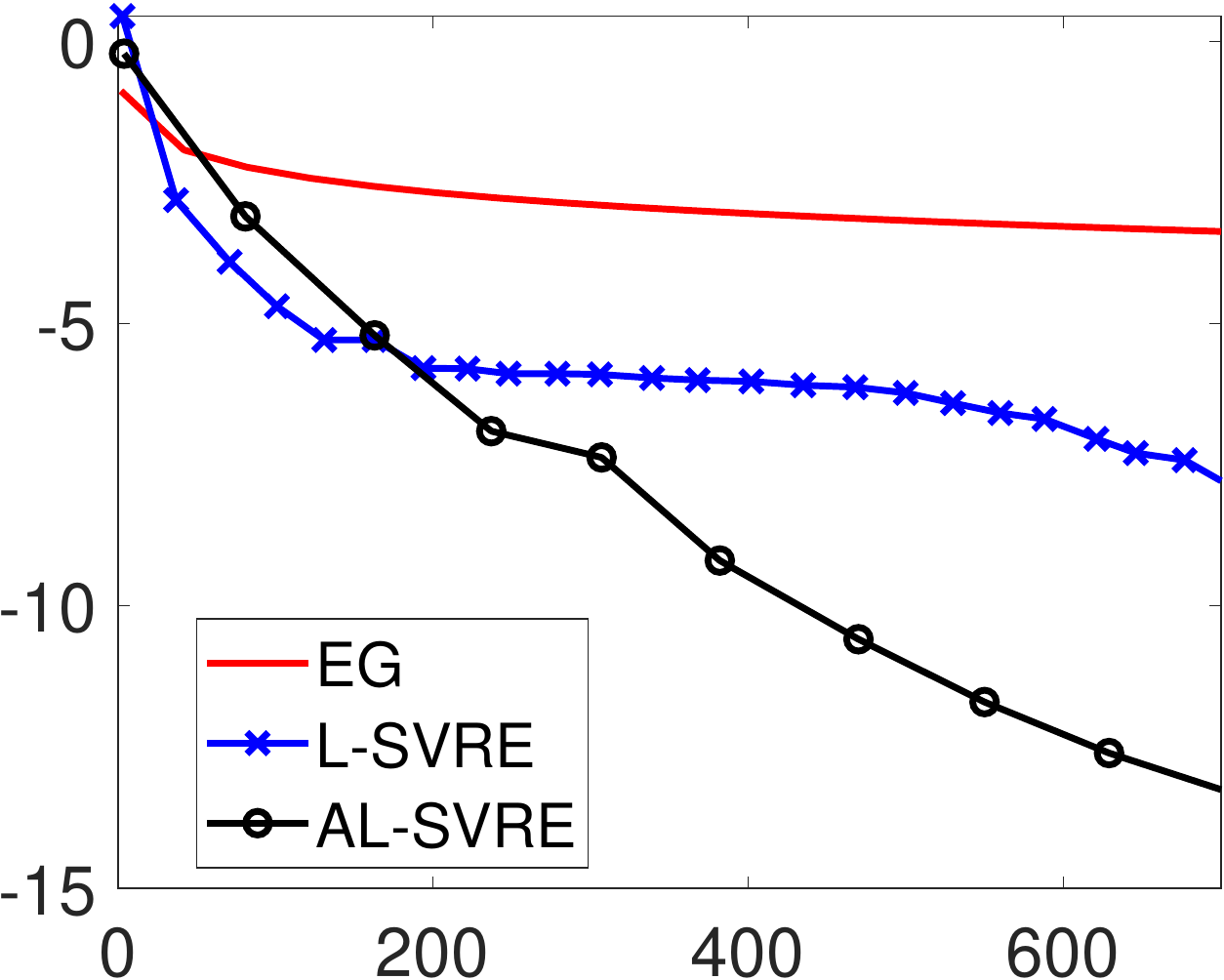} \\[0.1cm]
\small (a) AUC, a9a & \small  (b) AUC, w8a & 
\small  (c) wireless-500 & \small (d) wireless-1000
\end{tabular}
\caption{We demonstrate comparison on AUC maximization for SCSC minimax in sub-figures (a) and (b); and wireless communication for general convex-concave minimax in sub-figures (c) and (d).}\vskip-0.2cm\label{figure:experiment}
\end{figure*}

\begin{cor}\label{cor:csc}
Under settings of Lemma \ref{lem:csc}, for $\eps \le 4 L D_x^2$, the total complexity of SFO calls for finding an $\eps/2$-saddle-point of $f_{\eps, \vx_0}$, which is also an $\eps$-saddle-point of $f$, is 
\begin{align*}
\mbox{\small$\displaystyle{
\tilde\fO\left(\!\left(n\!+\!D_x \sqrt{\frac{n L \kappa_y}{\eps}}\!+\!n^{3/4}\sqrt{\kappa_y}\!+\!n^{3/4}D_x \sqrt{\frac{L}{\eps}}\right)\log\left(\frac{1}{\eps}\right)\!\right)}$}
\end{align*}
\end{cor}
Recently, \citet{yang2020catalyst} also studied Catalyst acceleration for strongly-convex-concave minimax optimization, but they only considered the special case of $\eps<\mu_y$ which leads to that their result does not match the lower bound.


\subsection{Convex-Concave Case}
In general convex-concave setting, we suppose $\fX$ and $\fY$ are bounded with diameters $D_x$ and $D_y$ respectively. Then unbalance of the problem comes from the difference between diameters $D_x$ and $D_y$.
We introduce the auxiliary SCSC function as follows:
\begin{align}\label{fun:cc}
\begin{split}
 f_{\eps, \vx_0, \vy_0} (\vx, \vy) 
\triangleq f(\vx, \vy) + \frac{\eps}{8 D_x^2} \norm{\vx - \vx_0}^2 - \frac{\eps}{8 D_y^2} \norm{\vy - \vy_0}^2.
\end{split}
\end{align}
Similarly, Lemma \ref{lem:cc} shows that any $\eps/2$-saddle-point of $f_{\eps, \vx_0, \vy_0}$ is an $\eps$-saddle-point of $f$. 
\begin{lem}\label{lem:cc}
Suppose that $f(\vx, \vy)$ is convex-concave, $\fX$ and $\fY$ are bounded with diameters $D_x$ and $D_y$ respectively. Consider the function
$f_{\eps, \vx_0, \vy_0} (\vx, \vy)$ defined as (\ref{fun:cc}).
Then for any $(\hat\vx, \hat\vy) \in \fX \times \fY$, we have
\begin{align*}
\max_{\vy \in \fY} f (\hat{\vx}, \vy) - \min_{\vx \in \fX} f (\vx, \hat{\vy}) 
\le \frac{\eps}{2} + \max_{\vy \in \fY} f_{\eps, \vx_0, \vy_0} (\hat{\vx}, \vy) - \min_{\vx \in \fX} f_{\eps, \vx_0, \vy_0} (\vx, \hat{\vy}).
\end{align*}
\end{lem}
Combining Lemma~\ref{lem:cc} and Theorem~\ref{thm:asvre}, we obtain the convergence result of Corollary \ref{cor:cc}.
\begin{cor}\label{cor:cc}
Under settings of Lemma \ref{lem:cc}, for any $\eps \le 4 L \min\{D_x^2, D_y^2\}$, the total complexity of SFO calls for finding an $\eps/2$-saddle-point of $f_{\eps, \vx_0, \vy_0}$, which is also an $\eps$-saddle-point of $f$, is 
\begin{align*}
\mbox{$\displaystyle{
\tilde\fO\!\left(\!\left(n\!+\!\frac{\sqrt{n} L D_x D_y}{\eps}\!+\!n^{3/4} (D_x\!+\!D_y) \sqrt{\frac{L}{\eps}}\right)\log\left(\frac{1}{\eps}\right)\!\right)}$}.
\end{align*}
\end{cor}
The result of Corollary~\ref{cor:cc} nearly matches the lower bound w.r.t.\ $\eps$, $L$, $n$, $D_x$ and $D_y$ simultaneously~\cite{han2021lower}. Note that the best known  upper bound in convex-concave setting~\cite{alacaoglu2021stochastic} is optimal to $\eps$, $L$ and $n$, but it regards the diameters as constants and does not consider the potential unbalance arose from $D_x$ and $D_y$.

\section{Experiments}

We conduct the experiments on AUC maximization~\cite{hanley1982meaning,ying2016stochastic,shen2018towards} and  wireless communication~\cite{garnaev2009eavesdropping,boyd2004convex,yang2020catalyst} problems. 
We evaluate the performance of AL-SVRE and compare it with baseline algorithms ExtraGradient (EG)~\cite{korpelevich1977extragradient,gidel2019variational} and L-SVRE~\cite{alacaoglu2021stochastic}. 
We summarize the datasets  in Table~\ref{table:dataset}.
The empirical results in Figure~\ref{figure:experiment} show that our proposed AL-SVRE performs better than the baselines.

\subsection{AUC Maximization}

AUC maximization~\cite{hanley1982meaning} considers finding the classifier $\vtheta\in\BR^d$ on training set $\{(\va_i, b_i)\}_{i=1}^n$, where $\va_i\in\BR^d$ and $b_i\in\{+1,-1\}$. We denote $n^+$ be the numbers of positive instances and let $\hat p=n^+/n$. Our experiments focus on the unbalanced SCSC minimax formulation for AUC maximization~\cite{shen2018towards,ying2016stochastic} as follows
\begin{align}\label{prob:AUC}
    \min_{\vx\in\BR^{d+2}}\max_{y\in\BR} f(\vx,y) \triangleq \frac{1}{n}\sum_{i=1}^n f_i(\vx,y;\va_i,b_i,\lambda),
\end{align}
where $\vx=[\vtheta;u;v]\in\BR^{d+2}$, $\lambda>0$ is the regularization parameter and each component is defined as
\begin{align*}
f_i(\vx,\vy;\va_i,b_i,\lambda)
= & \frac{\lambda}{2}\|\vx\|_2^2-\hat p(1-\hat p)y^2\\
&+\hat p((\vtheta^\top\va_i-v)^2+2(1+y)\vtheta^\top\va_i)\mathbb{I}_{b_i=-1}\\
& +(1-\hat p)((\vtheta^T\va_i-u)^2-2(1+y)\vtheta^\top\va_i)\mathbb{I}_{b_i=1}.
\end{align*}
We set $\lambda=10^{-10}$ and evaluate all algorithms on datasets ``a9a'' and ``w8a'' \cite{CC01a}. We tune stepsize of EG and L-SVRE (or as sub-problem solver) from $\{0.02, 0.05, 0.1, 0.2, 0.5\}$.
For L-SVRE, we let $p=1/2n$ by the setting of Theorem~\ref{algo:svre}.
For AL-SVRE, we let $\beta=0.01$, $q = \lambda_x/(\lambda + \beta)$; $T_k=\lceil 0.3n \rceil$ for ``a9a'' and $T_k=\lceil 0.5n \rceil$ for ``w8a''. We present the results of epochs against $\log\norm{\nabla f(\vx,\vy)}$ in Figure~\ref{figure:experiment} (a) and (b).

\subsection{The Wireless Communication Problems}

We consider the wireless communication problem~\cite{yang2020catalyst,boyd2004convex} as follows
\begin{align}\label{prob:wireless}
    \min_{\vx\in\fX}\max_{\vy\in\fY} f(\vx, \vy) \triangleq -\frac{1}{n}\sum_{i=1}^n \log\left(1+\frac{b_ix_i}{a_i+y_i}\right),
\end{align}
where constrained sets are $\fX=\{\vx: \norm{\vx} \leq R, x_i\geq 0\}$ and $\fY=\{\vy: {\bf 1}^\top\vy=n, y_i\geq 0\}$. We can verify the objective function of problem (\ref{prob:wireless}) is convex-concave but possibly be neither strongly-convex nor strongly-concave.

We generate two datasets with
(i) $n = 500$, $R = 1$, $\vb = \vone$, $\va \in \BR^{500}$ uniformly from $[0, 10]^{500}$ and (ii) $n = 1000$, $R = 1$, $\vb = \vone$, $\va \in \BR^{1000}$ uniformly from $[0, 50]^{1000}$. 
We tune the stepsize for all the algorithms from $[0.01, 0.1, 1.0]$ and set $p=2/n$ like AUC maximization. For AL-SVRE, we use the auxiliary function as follows
\begin{align*}
 \hat{f}(\vx, \vy) 
\triangleq  f(\vx, \vy) + \frac{\eps}{8 R^2} \norm{\vx}^2 - \frac{\eps}{8 n^2} \norm{\vy - \vone}^2  
\end{align*}
where $\eps=10^{-6}/n$. We also set $T_k=500$, $q=R^2/n^2$ and $\beta = \eps/(4 R^2) - \eps/(4 n^2)$.
Since the problem (\ref{prob:wireless}) is constrained, we use evaluation the performance by epochs against the logarithm of the magnitude of gradient mapping, which is defined as 
\begin{align*}
\mbox{\small
$\displaystyle{\frac{\norm{\vx-\fP_{\fX} (\vx\!-\!\hat\tau \nabla_{\vx} f(\vx, \vy))} + \norm{\vy-\fP_{\fY} (\vy\!+\!\hat\tau \nabla_{\vy} f(\vx, \vy))}}{\hat\tau}}$}
\end{align*}
where $\hat\tau$ is set to be $0.1$. We present the empirical results in Figure~\ref{figure:experiment} (c) and (d).

\begin{table}
\caption{Summary of datasets}\label{table:dataset}
\centering\renewcommand{\arraystretch}{1.05}\vskip0.1cm
\begin{tabular}{|c|c|c|c|c|}
\hline
   &  a9a & w8a & wireless-500 & wireless-1000  \\\hline
  $n$  &  32,561  & 45,546 & 500 & 1,000 \\\hline
  $d$  &  ~~~~~123 & ~~~~~300 & 500 & 1,000 \\\hline
\end{tabular}
\end{table}

\section{Conclusions}

In this paper we have studied unbalanced convex-concave minimax problems with finite-sum structure. We have shown the optimality of L-SVRE for balanced SCSC minimax and proposed a near optimal algorithm AL-SVRE for unbalanced problems. AL-SVRE only contains two loops of iterations, making its implementation  more simple and practical than the existing methods for unbalanced SCSC minimax. We have also extended our algorithm to more general convex-concave minimax problems and showed the  near optimality.

It would be interesting to apply our ideas to solve specific convex-concave minimax with refined smoothness assumptions and specific bilinear settings~\cite{wang2020improved,xie2021dippa,ouyang2018lower}. 
It is also possible to design accelerated variance reduced algorithms to solve unbalanced minimax problems in online setting.

\bibliographystyle{plainnat}
\bibliography{reference}

\newpage
\onecolumn 
\appendix
 
\section*{Supplemental Materials}

The supplemental materials are organized as follows. Section~\ref{appendix:lemma} presents several lemmas for further theoretical analysis.
Section~\ref{appendix:LSVRE} provides the proof of Theorem \ref{thm:svre}, which is the convergence result of L-SVRE for balanced SCSC setting.
Section~\ref{app:ASVRE} gives the detailed analysis for AL-SVRE, including the proof of our main results Theorem \ref{thm:asvre}, Corollary~\ref{cor:ASVRE} and other lemmas.
Section~\ref{appendix:lower} provide the proof of Theorem \ref{thm:balance-lower-bound}, which implies L-SVRE is an optimal SFO algorithm for balanced finite-sum SCSC minimax. Section~\ref{appendix:csc} provides the detailed proofs for the results in Section~\ref{sec:extensions}. 


 \input{appendix/lemma}
 \input{appendix/svre}
 \input{appendix/asvre}
 \input{appendix/lower}
 \input{appendix/cc}

\end{document}

%% file: math.tex
\usepackage{amsmath}\allowdisplaybreaks
\usepackage{amsfonts,bm}
\usepackage{amssymb}

\def\eqref#1{equation~(\ref{#1})}

\def\ceil#1{\left\lceil #1 \right\rceil}
\def\floor#1{\left\lfloor #1 \right\rfloor}
\def\1{\bf{1}}

\newcommand{\Norm}[1]{\left\| #1 \right\|}
\newcommand{\norm}[1]{\left\| #1 \right\|_2}
\def\inner#1#2{\langle #1, #2 \rangle}

\def\eps{{\varepsilon}}

\def\vzero{{\bf{0}}}
\def\vone{{\bf{1}}}

\def\vtheta{{\bm{\theta}}}
\def\va{{\bf{a}}}
\def\vb{{\bf{b}}}
\def\vc{{\bf{c}}}

\def\ve{{\bf{e}}}

\def\vg{{\bf{g}}}

\def\vu{{\bf{u}}}
\def\vv{{\bf{v}}}
\def\vw{{\bf{w}}}
\def\vx{{\bf{x}}}
\def\vy{{\bf{y}}}
\def\vz{{\bf{z}}}

\def\fA{{\mathcal{A}}}

\def\fC{{\mathcal{C}}}

\def\fF{{\mathcal{F}}}

\def\fO{{\mathcal{O}}}
\def\fP{{\mathcal{P}}}

\def\fX{{\mathcal{X}}}
\def\fY{{\mathcal{Y}}}
\def\fZ{{\mathcal{Z}}}

\def\BE{{\mathbb{E}}}

\def\BR{{\mathbb{R}}}

\def\mB {{\bf B}}

\def\mI {{\bf I}}

\def\mU {{\bf U}}

\newcommand{\E}{\mathbb{E}}

\DeclareMathOperator*{\argmax}{arg\,max}
\DeclareMathOperator*{\argmin}{arg\,min}

\DeclareMathOperator{\spn}{span}

\usepackage{amsthm}
\theoremstyle{plain}
\newtheorem{thm}{Theorem}
\newtheorem{dfn}{Definition}

\newtheorem{lem}{Lemma}

\newtheorem{cor}{Corollary}

\makeatletter
\def\Ddots{\mathinner{\mkern1mu\raise\p@
\vbox{\kern7\p@\hbox{.}}\mkern2mu
\raise4\p@\hbox{.}\mkern2mu\raise7\p@\hbox{.}\mkern1mu}}
\makeatother

\makeatletter
\newcommand*{\rom}[1]{\expandafter\@slowromancap\romannumeral #1@}
\makeatother

%% file: appendix/lemma.tex
\section{Technical Lemmas}\label{appendix:lemma}
We first present some useful tools for the analysis of constrained optimization.

\begin{lem}[{\citet[Theorem 2.2.9 and 2.2.12]{nesterov2018lectures}}]\label{lem:proj-optimal}
Let $\vz^*=\argmin_{\vz\in\fC}Q(\vz)$, where $Q$ is smooth and strongly-convex; $\fC$ is convex and compact. Then, for any $\vz\in\fC$ and $\eta>0$, we have 
\begin{align*}
\inner{\nabla Q(\vz^*)}{\vz - \vz^*} \ge 0 \quad \text{and} \quad
\vz^* = \fP_{\fC} \left( \vz^* - \eta \nabla Q(\vz^*) \right).
\end{align*}
\end{lem}

\begin{lem}[{\citet[Corollary 2.2.3]{nesterov2018lectures}}]\label{lem:non-expansive}
Given a convex and compact set $\fC\subseteq\BR^{d}$; and any $\vu,\vv\in\BR^{d}$, we have $\norm{\fP_\fC(\vu)-\fP_\fC(\vv)}\leq\norm{\vu-\vv}$.
\end{lem}

\begin{lem}\label{lemma:proj-inner}
Given a convex and compact set $\fC\subseteq\BR^{d}$, for any $\vu\in\BR^d$ and $\vv\in\fC$, we have \[\inner{\fP_\fC(\vu)-\vu}{\fP_\fC(\vu)-\vv}\leq 0.\]
\end{lem}
\begin{proof}
Let $Q(\vz)=\frac{1}{2}\norm{\vz-\vu}^2$, then $\fP_\fC(\vu)=\argmin_{\vz\in\fC} Q(\vz)$. Using Lemma~\ref{lem:proj-optimal}, we have
\begin{align*}
\inner{\fP_\fC(\vu)-\vu}{\vv-\fP_\fC(\vu)}
=\inner{\nabla Q(\fP_\fC(\vu))}{\vv-\fP_\fC(\vu)} 
\geq 0.
\end{align*}
\end{proof}

Then we provides some properties for convex-concave functions.
\begin{lem}[{\citet[Lemma B.2]{lin2020near}}]\label{lem:phi-psi}
Assume that $f(\vx, \vy): \fX \times \fY \to \BR$ is $L$-smooth and $(\mu_x, \mu_y)$-convex-concave. 
We define
\begin{align*}
     \vy_f^*(\cdot) &\triangleq \argmax_{\vy \in \fY} f(\cdot, \vy), ~
       \Phi_f(\cdot) \triangleq \max_{\vy \in \fY} f(\cdot, \vy), \\
  \vx_f^*(\cdot) &\triangleq \argmin_{\vx \in \fX} f(\vx, \cdot), ~~
  \Psi_f(\cdot) \triangleq \min_{\vx \in \fX} f(\vx, \cdot).
\end{align*}
Then, there holds that
\begin{enumerate}
    \item[\emph{(a)}] the function $\vy_f^*(\cdot)$ is $\kappa_y$-Lipschitz, 
    \item[\emph{(b)}] the function $\Phi_f(\cdot)$ is $2\kappa_yL$-smooth and $\mu_x$-strongly convex with $\nabla \Phi_f(\cdot) = \nabla_{\vx} f(\cdot, \vy_f^*(\cdot))$,
    \item[\emph{(c)}] the function $\vx_f^*(\cdot)$ is $\kappa_x$-Lipschitz,
    \item[\emph{(d)}] the function $\Psi_f(\cdot)$ is $2\kappa_xL$-smooth and $\mu_y$-strongly concave with $\nabla \Psi_f(\cdot) = \nabla_{\vx} f(\vx_f^*(\cdot), \cdot)$.
\end{enumerate}
\end{lem}

\begin{lem}\label{lem:monotone}
Assume that $f(\vx, \vy): \fX \times \fY \to \BR$ is differentiable and $(\mu_x, \mu_y)$-convex-concave, then for any $\vz_1=(\vx_1,\vy_1), \vz_2=(\vx_2,\vy_2) \in \fX \times \fY$, we have
\begin{align}
    \inner{\vg(\vz_1) - \vg(\vz_2)}{\vz_1 - \vz_2} \ge \mu_x \norm{\vx_1 - \vx_2}^2 + \mu_y \norm{\vy_1 - \vy_2}^2.
\end{align}
where $g(\vx,\vy)=(\nabla_\vx f(\vx,\vy), -\nabla_\vy f(\vx,\vy))$.
\end{lem}

\begin{proof}
According to $f(\cdot, \vy)$ is $\mu_x$-strongly-convex, we have
\begin{align*}
    f(\vx_2, \vy_1) - f(\vx_1, \vy_1) &\ge \inner{\nabla_{\vx} f(\vx_1, \vy_1)}{\vx_2 - \vx_1} + \frac{\mu_x}{2} \norm{\vx_2 - \vx_1}, \\
    f(\vx_1, \vy_2) - f(\vx_2, \vy_2) &\ge \inner{\nabla_{\vx} f(\vx_2, \vy_2)}{\vx_1 - \vx_2} + \frac{\mu_x}{2} \norm{\vx_2 - \vx_1}. 
\end{align*}
Similarly, by strongly-convexity of $-f(\vx, \cdot)$, we have
\begin{align*}
    -f(\vx_1, \vy_2) + f(\vx_1, \vy_1) &\ge -\inner{\nabla_{\vy} f(\vx_1, \vy_1)}{\vy_2 - \vy_1} + \frac{\mu_y}{2} \norm{\vy_2 - \vy_1}, \\
    -f(\vx_2, \vy_1) + f(\vx_2, \vy_2) &\ge -\inner{\nabla_{\vy} f(\vx_2, \vy_2)}{\vy_1 - \vy_2} + \frac{\mu_y}{2} \norm{\vy_2 - \vy_1}. 
\end{align*}
The desired result just follows from adding above four inequalities together. 
\end{proof}

The following two lemmas show that the relationship between the smoothness of each component $f_i$, the average-smoothness of $\{f_i\}_{i=1}^n$ and the smoothness of $\frac{1}{n} \sum_{i=1}^n f_i$.

\begin{lem}\label{lem:smooth-average-smooth}
Assume that $f_i$ is $L_i$-smooth, then $\{f_i\}_{i=1}^n$ is $\sqrt{\frac{1}{n} \sum_{i=1}^n L_i^2}$-average-smooth.
\end{lem}
\begin{proof}
    By the definition of smoothness, we have
    \begin{align*}
        \frac{1}{n} \sum_{i=1}^n \norm{f_i(\vz_1) - f_i(\vz_2)}^2 \le \frac{1}{n} \sum_{i=1}^n L_i^2 \norm{\vz_1 - \vz_2}^2,
    \end{align*}
    which is our desired result. 
\end{proof}

\begin{lem}\label{lem:average-smooth}
    Assume that $\{f_i\}_{i=1}^n$ is $L$-average-smooth, then $f(\vz) = \frac{1}{n} f_i(\vz)$ is $L$-smooth.
\end{lem}
\begin{proof}
    Note that
    \begin{align*}
        \norm{\nabla f(\vz_1) - \nabla f(\vz_2)} &= \norm{\frac{1}{n} \sum_{i=1}^n (\nabla f_i(\vz_1) - \nabla f_i(\vz_2))} \\
        &\le \frac{1}{n} \sum_{i=1}^n \norm{\nabla f_i(\vz_1) - \nabla f_i(\vz_2)} \\
        &\le \sqrt{\frac{1}{n} \sum_{i=1}^n \norm{\nabla f_i(\vz_1) - \nabla f_i(\vz_2)}^2} \\
        &\le \sqrt{L^2 \norm{\vz_1 - \vz_2}^2} = L \norm{\vz_1 - \vz_2},
    \end{align*}
    where the first inequality is according to triangle inequality, the second inequality follows from AM-QM inequality, and the last inequality is due to average-smoothness of $\{f_i\}_{i=1}^n$.
\end{proof}

%% file: appendix/svre.tex
\section{Proof of Theorem \ref{thm:svre}}\label{appendix:LSVRE}
Following proof is adapted from the proof of Theorem 4.9 in \cite{alacaoglu2021stochastic}.
\begin{proof}
    By Lemma~\ref{lemma:proj-inner}, we know that
    \begin{align*}
        \inner{\vz_{k+1/2} - \bar{\vz}_k + \tau \vg(\vw_k)}{\vz_{k+1} - \vz_{k+1/2}} \ge 0, \\
        \inner{\vz_{k+1} - \bar{\vz}_k + \tau [\vg(\vw_k) + \vg_i(\vz_{k+1/2}) - \vg_i(\vw_k)]}{\vz^* - \vz_{k+1}} \ge 0.
    \end{align*}
    Summing above two inequalities, we have
    \begin{equation}\label{eq:1}
    \begin{aligned}
        \inner{\vz_{k+1} - \bar{\vz}_k}{\vz^* - \vz_{k+1}} + \inner{\vz_{k+1/2} - \bar{\vz}_k}{\vz_{k+1} - \vz_{k+1/2}} \\
        + \tau \inner{\vg(\vw_k) + \vg_i(\vz_{k+1/2}) - \vg_i(\vw_k)}{\vz^* - \vz_{k+1/2}} \\
        + \tau \inner{\vg_i(\vz_{k+1/2}) - \vg_i(\vw_k)}{\vz_{k+1/2} - \vz_{k+1}} \ge 0.
    \end{aligned}
    \end{equation}
    
    Since $2\inner{\va}{\vb} = \norm{\va + \vb}^2 - \norm{\va}^2 - \norm{\vb}^2$, the first term in inequality (\ref{eq:1}) can be written as
    \begin{equation}\label{eq:1-1}
    \begin{aligned}
        & 2\inner{\vz_{k+1} - \bar{\vz}_k}{\vz^* - \vz_{k+1}} \\
        =& 2 \inner{\vz_{k+1} - \alpha \vz_k - (1 - \alpha) \vw_k}{\vz^* - \vz_{k + 1}} \\
        =& 2\alpha \inner{\vz_{k+1} - \vz_k}{\vz^* - \vz_{k+1}} + 2(1 - \alpha) \inner{\vz_{k+1} - \vw_k}{\vz^* - \vz_{k+1}} \\
        =& \alpha \big( \norm{\vz_k - \vz^*}^2 - \norm{\vz_{k+1} - \vz^*}^2 - \norm{\vz_{k+1} - \vz_k}^2 \big) \\
        & + (1 - \alpha) \big( \norm{\vw_k - \vz^*}^2 - \norm{\vz_{k+1} - \vz^*}^2 - \norm{\vz_{k+1} - \vw_k}^2 \big) \\
        =& \alpha \norm{\vz_k - \vz^*}^2 + (1 - \alpha) \norm{\vw_k - \vz^*}^2 - \norm{\vz_{k+1} - \vz^*}^2 \\
        & - \alpha \norm{\vz_{k+1} - \vz_k}^2 - (1 - \alpha) \norm{\vz_{k+1} - \vw_k}^2.
    \end{aligned}
    \end{equation}
    Similarly, the second term in inequality (\ref{eq:1}) can be written as
    \begin{equation}\label{eq:1-2}
    \begin{aligned}
        & 2\inner{\vz_{k+1/2} - \bar{\vz}_k}{\vz_{k+1} - \vz_{k+1/2}} \\
        =& \alpha \norm{\vz_k - \vz_{k+1}}^2 + (1 - \alpha) \norm{\vw_k - \vz_{k+1}}^2 - \norm{\vz_{k+1/2} - \vz_{k+1}}^2 \\
        & - \alpha \norm{\vz_{k+1/2} - \vz_k}^2 - (1 - \alpha) \norm{\vz_{k+1/2} - \vw_k}^2.
    \end{aligned}
    \end{equation}
    
    Using the fact $\E_k [\vg_i(\vz)] = \vg(\vz)$, the expectation of third term in inequality (\ref{eq:1}) can be bounded as 
    \begin{equation}\label{eq:1-3}
    \begin{aligned}
        &\quad 2 \E_k \left[ \inner{\vg(\vw_k) + \vg_i(\vz_{k+1/2}) - \vg_i(\vw_k)}{\vz^* - \vz_{k+1/2}} \right] \\
        &= 2 \inner{\vg(\vz_{k+1/2})}{\vz^* - \vz_{k+1/2}} \\
        &\le 2 \inner{\vg(\vz^*)}{\vz^* - \vz_{k+1/2}} - 2 \mu \norm{\vz_{k+1/2} - \vz^*}^2 \\
        &\le - \mu \E_k\left[\norm{\vz_{k+1} - \vz^*}^2\right] + 2 \mu \E_k\left[\norm{\vz_{k+1/2} - \vz_{k+1}}^2\right],
    \end{aligned}
    \end{equation}
    where the first inequality follows from Lemma \ref{lem:monotone}, and the last inequality is according to Lemma \ref{lem:proj-optimal} and $\norm{\va + \vb}^2 \le 2 \norm{\va}^2 + 2 \norm{\vb}^2$.
    
    Moreover, by Young's inequality $2\inner{\va}{\vb} \le \beta \norm{\va}^2 + \frac{1}{\beta} \norm{\vb}^2$ and $L$-average-smoothness of $\{f_i\}_{i=1}^n$, there holds
    \begin{equation}\label{eq:1-4}
    \begin{aligned}
        &\quad \E_k \left[ 2\tau \inner{\vg_i(\vz_{k+1/2}) - \vg_i(\vw_k)}{\vz_{k+1/2} - \vz_{k+1}} \right] \\
        &\le 2\tau^2 \E_k \left[ \norm{\vg_i(\vz_{k+1/2}) - \vg_i(\vw_k)}^2 \right] + \frac{1}{2} \E_k \left[ \norm{{\vz_{k+1/2} - \vz_{k+1}}}^2\right] \\
        &\le 2\tau^2 L^2  \norm{\vz_{k+1/2} - \vw_k}^2 + \frac{1}{2} \E_k \left[ \norm{{\vz_{k+1/2} - \vz_{k+1}}}^2\right].
    \end{aligned}
    \end{equation}

    Plugging results of (\ref{eq:1-1}), (\ref{eq:1-2}), (\ref{eq:1-3}) and (\ref{eq:1-4}) into inequality (\ref{eq:1}), we obtain that
    \begin{align*}
        &\quad \alpha \norm{\vz_k - \vz^*}^2 + (1 - \alpha) \norm{\vw_k - \vz^*}^2 - \E_k \left[\norm{\vz_{k+1} - \vz^*}^2\right] - \E_k \left[\norm{\vz_{k+1/2} - \vz_{k+1}}^2\right] \\
        &- \alpha \norm{\vz_{k+1/2} - \vz_k}^2 - (1 - \alpha) \norm{\vz_{k+1/2} - \vw_k}^2 \\
        &- \tau\mu \E_k \left[\norm{\vz_{k+1} - \vz^*}^2\right] + 2\tau\mu \E_k \left[\norm{\vz_{k+1/2} - \vz_{k+1}}^2 \right] \\
        &+ 2\tau^2 L^2 \norm{\vz_{k+1/2} - \vw_k}^2 + \frac{1}{2} \E_k \left[ \norm{{\vz_{k+1/2} - \vz_{k+1}}}^2\right] \ge 0,
    \end{align*}
    that is 
    \begin{align*}
        & (1 + \tau \mu) \E_k \left[\norm{\vz_{k+1} - \vz^*}^2\right] \\
        \le & \alpha \norm{\vz_k - \vz^*}^2 + (1 - \alpha) \norm{\vw_k - \vz^*}^2 
        - \left(\frac{1}{2} - 2 \tau \mu\right) \E_k \left[ \norm{{\vz_{k+1/2} - \vz_{k+1}}}^2\right] \\
        &- (1 - \alpha - 2 \tau^2 L^2) \norm{\vz_{k+1/2} - \vw_k}^2.
    \end{align*}

Consequently, with setting $\alpha = 1 - p$, $p = \frac{1}{2 n}$ and $\tau = \frac{1}{4\sqrt{n} L}$, we know that $\frac{1}{2} - 2 \tau \mu \ge 0$, $1 - \alpha - 2\tau^2 L^2 \ge 0$ and
\begin{align}\label{eq:2-1}
    (1 + \tau \mu) \E_k \left[\norm{\vz_{k+1} - \vz^*}^2\right] &\le (1 - p) \norm{\vz_k - \vz^*}^2 + p \norm{\vw_k - \vz^*}^2.
\end{align}

On the other hand, by definition of $\vw_{k+1}$, we have
\begin{align}\label{eq:2-2}
    \E_{k} \left[ \norm{\vw_{k+1} - \vz^*}^2 \right] = (1 - p) \norm{\vw_k - \vz^*}^2 + p \E_{k}\left[\norm{\vz_{k+1} - \vz^*}^2\right].
\end{align}

Following from results of (\ref{eq:2-1}) and (\ref{eq:2-2}), there holds
\begin{align*}
    &\quad (1 + \tau \mu) \E_k \left[\norm{\vz_{k+1} - \vz^*}^2\right] + c \E_{k} \left[ \norm{\vw_{k+1} - \vz^*}^2 \right] \\
    &\le (1 - p) \norm{\vz_k - \vz^*}^2 + p \norm{\vw_k - \vz^*}^2 + c(1 - p) \norm{\vw_k - \vz^*}^2 + c p \E_{k} \left[\norm{\vz_{k+1} - \vz^*}^2\right],
\end{align*}
that is
\begin{equation}\label{eq:3}
\begin{aligned}
    &\quad (1 + \tau \mu - c p) \E_k \left[\norm{\vz_{k+1} - \vz^*}^2\right] + c \E_{k} \left[ \norm{\vw_{k+1} - \vz^*}^2 \right] \\
    &\le (1 - p) \norm{\vz_k - \vz^*}^2 + (p + c (1 - p)) \norm{\vw_k - \vz^*}^2.
\end{aligned}
\end{equation}

Letting $c = \frac{2\tau \mu + 2 p}{\tau \mu + 2 p}$ and noticing that $\tau \mu < 1, p < 1$, we have
\begin{align*}
    \frac{1 - p}{1 + \tau \mu - c p} &= 1 - \frac{\tau \mu - p(c - 1)}{1 + \tau \mu - c p}  
    = 1 - \frac{\tau \mu - \frac{p \tau \mu}{\tau \mu + 2 p}}{1 + \tau \mu - p \frac{2\tau \mu + 2 p}{\tau \mu + 2 p}} \\
    &= 1 - \frac{\tau^2 \mu^2 + p \tau \mu}{\tau \mu(1 + \tau \mu) + 2p (1 - p)} \\
    &\le 1 - \frac{\tau \mu(\tau\mu + p)}{2 \tau \mu + 2p} \le 1 - \frac{p \tau \mu}{2 (\tau \mu + p)}.
\end{align*}
and
\begin{align*}
    \frac{p + c(1 - p)}{c} = 1 - p\left(1 - \frac{\tau \mu + 2 p}{2\tau \mu + 2 p}\right) = 1 - \frac{p \tau \mu}{2 (\tau \mu + p)}.
\end{align*}
Together with inequality (\ref{eq:3}), there holds
\begin{align*}
    &\quad (1 + \tau \mu - c p) \E_k \left[\norm{\vz_{k+1} - \vz^*}^2\right] + c \E_{k} \left[ \norm{\vw_{k+1} - \vz^*}^2 \right] \\
    &\le \left(1 - \frac{p \tau \mu}{2 (\tau \mu + p)}\right)\left( (1 + \tau \mu- c p) \norm{\vz_k - \vz^*}^2 + c \norm{\vw_k - \vz^*}^2\right),
\end{align*}
which implies that 
\begin{align*}
    &\quad \left( \tau \mu + p \right) \E \left[\norm{\vz_{k} - \vz^*}^2\right] \\
    &\le \left( \tau \mu (1 + \tau \mu) + 2p(1 - p) \right) \E \left[\norm{\vz_{k} - \vz^*}^2\right] + 2\left( \tau \mu + p \right) \left[ \norm{\vw_{k} - \vz^*}^2 \right] \\
    &\le \left(1 - \frac{p \tau \mu}{2 (\tau \mu + p)}\right)^{k} \left( \left( \tau \mu (1 + \tau \mu) + 2p(1 - p) \right) \E \left[\norm{\vz_{0} - \vz^*}^2\right] + 2\left( \tau \mu + p \right) \left[ \norm{\vw_{0} - \vz^*}^2 \right] \right) \\
    &\le 4 \left( \tau \mu + p \right)  \norm{\vz_0 - \vz^*}^2 \left(1 - \frac{p \tau \mu}{2 (\tau \mu + p)}\right)^{k},
\end{align*}
where we have recalled that $2(1 - p) \ge 1$ due to $p = \frac{1}{2n} \le \frac{1}{2}$ and $\tau \mu < 1$.

\end{proof}

%% file: appendix/asvre.tex
\section{Convergence Analysis of AL-SVRE}\label{app:ASVRE}

In this section, we aim to give the convergence rate of AL-SVRE.
We first give the proof of Lemma \ref{lem:SVRE-prime-dual} and establish the connection between the distance to saddle point and primal-dual gap. Then, we show upper bounds of some auxiliary quantities, which is useful to the analysis for AL-SVRE. Finally, we provide the formal proofs of Theorem \ref{thm:asvre} and Corollary~\ref{cor:ASVRE}.

\subsection{The proof of Lemma \ref{lem:SVRE-prime-dual}}
\begin{proof}
    The definition of $\tilde\vx$ means
    \begin{align*}
        \tilde\vx = \argmin_{\vx \in \fX} \left(\inner{\nabla_{\vx} f(\hat\vx, \hat\vy)}{\vx - \hat\vx} + \frac{1}{2\eta} \norm{\vx - \hat\vx}^2\right).
    \end{align*}
    Therefore, for any $\vx \in \fX$, we have
    \begin{align*}
        \inner{\nabla_{\vx} f(\hat\vx, \hat\vy)}{\vx - \hat\vx} + \frac{1}{2\eta} \norm{\vx - \hat\vx}^2\ge \inner{\nabla_{\vx} f(\hat\vx, \hat\vy)}{\tilde\vx - \hat\vx} + \frac{1}{2\eta} \norm{\tilde\vx - \hat\vx}^2,
    \end{align*}
    that is 
    \begin{align}\label{eq:hat-f-sc}
        \inner{\nabla_{\vx} f(\hat\vx, \hat\vy)}{\vx - \tilde\vx} &\ge \frac{1}{2\eta} \left(\norm{\tilde\vx - \hat\vx}^2 - \norm{\vx - \hat\vx}^2\right)
        \ge - \frac{1}{2\eta} \norm{\vx - \hat\vx}^2.
    \end{align}
    
    Let $\Phi_f(\vx)=\max_{\vy\in\fY}f(\vx,\vy)$ and $\vy_f^*(\vx)=\argmax_{\vy\in\fY}f(\vx,\vy)$, then it holds that
    \begin{align}\label{eq:svre-1}
    \begin{split}
         &\quad \Phi_f(\tilde\vx) - \Phi_f(\vx^*) \\
         &= \Phi_f(\tilde\vx) - \Phi_f(\hat\vx) - (\Phi_f(\vx^*) - \Phi_f(\hat\vx)) \\
         &\le \inner{\nabla \Phi_f(\hat\vx)}{\tilde\vx - \hat\vx} + \kappa_y L \norm{\tilde\vx - \hat\vx}^2 - \inner{\nabla \Phi_f(\hat\vx)}{\vx^* - \hat\vx} \\
         &= \inner{\nabla_{\vx} f(\hat\vx, \vy_f^*(\hat\vx))}{\tilde\vx - \vx^*} + \kappa_y L \norm{\tilde\vx - \hat\vx}^2 \\
         &= \inner{\nabla_{\vx} f(\hat\vx, \vy_f^*(\hat\vx)) - \nabla_{\vx} f(\hat\vx, \hat\vy)}{\tilde\vx - \vx^*} + \inner{\nabla_{\vx} f(\hat\vx, \hat\vy)}{\tilde\vx - \vx^*} + \kappa_y L \norm{\tilde\vx - \hat\vx}^2 \\
         &\le L \norm{\hat\vy - \vy_f^*(\hat\vx)}\norm{\tilde\vx - \vx^*} + 
        \frac{1}{2\eta} \norm{\hat\vx - \vx^*}^2 + \kappa_y L \norm{\tilde\vx - \hat\vx}^2.
    \end{split}
    \end{align}
    where the first inequality is according to $\Phi_f$ is $2\kappa_yL$-smooth and convex, the last inequality is due to inequality (\ref{eq:hat-f-sc}) with setting $\vx = \vx^*$.
    
    According to the $\norm{\va + \vb}^2 \le 2 \norm{\va}^2 + 2 \norm{\vb}^2$ and the Lipschitz continuity of $\vy_f^*(\cdot)$ (cf. Lemma \ref{lem:phi-psi}), we observe that
    \begin{align}\label{eq:svre-2}
    \begin{split}
         \norm{\hat\vy - \vy_f^*(\hat\vx)}^2 &\le 2\norm{\hat\vy - \vy^*}^2 + 2\norm{\vy^* - \vy_f^*(\hat\vx)}^2 \\
         &\le 2 \norm{\hat\vy - \vy^*}^2 + 2 \kappa_y^2 \norm{\hat\vx - \vx^*} \\
         &\le 2 \kappa_y^2 \eps.
    \end{split}
    \end{align}
    
    Next, following from optimality of $\vx^*$ such that $\vx^*=\argmin_{\vx\in\fX}\Phi_f(\vx)$ and Lemma \ref{lem:proj-optimal}, we have
    \begin{align*}
        \vx^* = \fP_{\fX} \left( \vx^* - \eta \nabla_{\vx} \Phi(\vx^*) \right) = \fP_{\fX} \left(\vx^* - \eta \nabla_{\vx} f(\vx^*, \vy^*)\right).
    \end{align*}
    Hence, by smoothness of the function $f$ and Lemma \ref{lem:non-expansive}, we have
    \begin{align}\label{eq:svre-3} 
    \begin{split}
         \norm{\tilde\vx - \vx^*} 
         &= \norm{\fP_{\fX} \left(\hat\vx - \eta \nabla_{\vx} f(\hat\vx, \hat\vy)\right) - \fP_{\fX} \left(\vx^* - \eta \nabla_{\vx} f(\vx^*, \vy^*)\right)} \\
         &\le \norm{\hat\vx - \vx^* - \eta \left(\nabla_{\vx} f(\hat\vx, \hat\vy) - \nabla_{\vx} f(\vx^*, \vy^*)\right)} \\
        &\le \norm{\hat\vx - \vx^*} + \eta L \sqrt{\norm{\hat\vx - \vx^*}^2 + \norm{\hat\vy - \vy^*}^2}\\
        &\le (1 + \eta L) \sqrt{\eps}.
    \end{split}
    \end{align}
    
    Consequently, plugging inequalities (\ref{eq:svre-2}) and (\ref{eq:svre-3}) into (\ref{eq:svre-1}) yields that 
    \begin{align*}
        \max_{\vy \in \fY} f(\tilde\vx, \vy) - f(\vx^*, \vy^*) 
        &\le \sqrt{2} (1 + \eta L) \kappa_y L \eps + \frac{\eps}{2\eta} + 2\kappa_y L \left( \norm{\tilde\vx - \vx^*}^2 + \norm{\hat\vx - \vx^*}^2 \right) \\
        &\le \left( \sqrt{2}(1 + \eta L) + 2(1 + \eta L)^2 + 2 \right) \kappa_y L \eps + \frac{\eps}{2\eta}
    \end{align*}
    
    Similarly, by definition of $\tilde\vy$, we also have
    \begin{align}\label{eq:svre-3b} 
         \norm{\tilde\vy - \vy^*} 
         &= \norm{\hat\vy - \vy^* + \eta \left(\nabla_{\vy} f(\hat\vx, \hat\vy) - \nabla_{\vy} f(\vx^*, \vy^*)\right)},
    \end{align}
    and 
    \begin{align*}
        g(\vx^*, \vy^*) - \min_{\vx \in \fX} g(\vx, \tilde\vy) \le \left( \sqrt{2}(1 + \eta L) + 2(1 + \eta L)^2 + 2 \right) \kappa_x L \eps + \frac{\eps}{2\eta}.
    \end{align*}
    
    Furthermore, by inequality (\ref{eq:svre-3}) and (\ref{eq:svre-3b}), we have
    {\small\begin{align}
        \notag & \norm{\tilde\vx - \vx^*}^2 + \norm{\tilde\vy - \vy^*}^2 \\
        \notag = & \norm{\fP_\fX\left(\hat\vx - \eta\nabla_{\vx} f(\hat\vx, \hat\vy)\right) - \fP_\fX\left(\vx^*-\eta\nabla_{\vx} f(\vx^*, \vy^*)\right)}^2 \\
        \notag & + \norm{\fP_\fY\left(\hat\vy + \eta\nabla_{\vy} f(\hat\vx, \hat\vy)\right) - \fP_\fY\left(\vy^* + \eta \nabla_{\vy} f(\vx^*, \vy^*)\right)}^2 \\
        \notag \le & 2\norm{\hat\vx - \vx^*}^2 + 2\norm{\hat\vy - \vy^*}^2 + 2\eta^2\left( \norm{\nabla_{\vx} f(\hat\vx, \hat\vy) - \nabla_{\vx} f(\vx^*, \vy^*)}^2 + \norm{\nabla_{\vy} f(\hat\vx, \hat\vy) - \nabla_{\vy} f(\vx^*, \vy^*)}^2\right) \\
        \label{eq:prime-dual-distance} \le & 2\eps + 2\eta^2 L^2 \left(\norm{\hat\vx - \vx^*}^2 + \norm{\hat\vy - \vy^*}^2\right) \le 2(1 + \eta^2L^2)\eps,
    \end{align}}
    where the second inequality is according to the smoothness of $f$.
\end{proof}

\subsection{The Proof of Lemma \ref{lem:dist-prime-dual}}

\begin{proof}
    Following the convexity of function $f(\cdot, \vy^*)$, we know that
    \begin{align*}
        \frac{\mu_x}{2} \norm{\hat\vx - \vx^*}^2 \le f(\hat\vx, \vy^*) - f(\vx^*, \vy^*). 
    \end{align*}
    Similarly, the concavity of function $f(\vx^*, \cdot)$ leads to
    \begin{align*}
        \frac{\mu_y}{2} \norm{\hat\vy - \vy^*}^2 \le f(\vx^*, \vy^*) - f(\vx^*, \hat\vy).
    \end{align*}
    Together with these pieces, it holds that
    \begin{align*}
        \mu_x \norm{\hat\vx - \vx^*}^2 + \mu_y \norm{\hat\vy - \vy^*}^2 \le 2\left( f(\hat\vx, \vy^*) - f(\vx^*, \hat\vy) \right) \le 2 \left( \max_{\vy \in \fY} f(\hat\vx, \vy) - \min_{\vx \in \fX} f(\vx, \hat\vy)\right).
    \end{align*}
\end{proof}

\subsection{The Proof of Lemma \ref{lem:aux}}


\begin{proof}
    It is easy to verify that that $F_k$ is $(L + \beta)$-smooth and $(\mu_x + \beta, \mu_y)$-convex-concave. 
    We will use induction to prove inequalities (\ref{eq:induction-2}) to (\ref{eq:induction-6}) hold for each $k \ge 1$.
    
    We first assume that inequalities (\ref{eq:induction-2}) to (\ref{eq:induction-6}) hold for any $k=1,2\dots, K-1$, then we prove the statements for $k=K$.
    
    \paragraph{Part (a): Inequalities (\ref{eq:induction-2}) and (\ref{eq:induction-3}) hold for $k = K$.} ~\\[0.05cm]
    By definition of $T_{K}$, we know that
    \begin{align*}
        e^{\theta T_{K}} \ge 12 \left(\frac{2}{1-\rho} + \frac{1728 \beta(L + \beta)(7(L+\beta) + 2\sqrt{n}\mu_y))}{\mu_x\mu_y \min\{\mu_x, \mu_y\} (1 - \rho)^2(\sqrt{q} - \rho)^2}\right) \triangleq R,
    \end{align*}
    where $\frac{1}{\theta} = 4\left(n + \frac{2 \sqrt{n}(L+\beta)}{\min\{\mu_x + \beta, \mu_y\}}\right)$.
    
    Let $\{\vu_{K, t}, \vv_{K, t}\}_{t \ge 0}$ be the sequence of using L-SVRE to solve minimax problem
    \begin{align*}
        \min_{\vu\in\fX}\max_{\vv\in\fY} F_{K}(\vu,\vv)    
    \end{align*}
    with initial point $(\vu_{K, 0}, \vv_{K, 0}) = (\vx_{K-1}, \vy_{K-1})$ and stepsize $\tau_{K} = \frac{1}{4\sqrt{n}(L + \beta)}$. 
    Based on Theorem \ref{thm:svre}, we have 
    \begin{align}\label{eq:inner-1}
    \begin{split}
        &\quad \E \left[\norm{\vu_{K, T_{K}} - \vx_{K}^*}^2 + \norm{\vv_{K, T_{K}} - \vy_{K}^*}^2 \right] \\ 
        &\le 4\left(1 - \theta\right)^{T_{K}} \E \left[\norm{\vx_{K-1} - \vx_{K}^*}^2 + \norm{\vy_{K-1} - \vy_{K}^*}^2\right] \\
        &\le 4 e^{-\theta T_{K}} \E \left[\norm{\vx_{K-1} - \vx_{K}^*}^2 + \norm{\vy_{K-1} - \vy_{K}^*}^2\right] \\
        &\le \frac{4}{R} \E \left[\norm{\vx_{K-1} - \vx_{K}^*}^2 + \norm{\vy_{K-1} - \vy_{K}^*}^2\right].
    \end{split}
    \end{align}
    
    Note that
    {\small\begin{align}\label{eq:inner-2}
    \begin{split}
         &\quad \E\left[\norm{\vx_{K-1} - \vx^*_{K}}^2 + \norm{\vy_{k-1} - \vy^*_{K}}^2\right] \\
         &\le 2\E\left[\norm{\vx_{K-1} - \vx^*_{K-1}}^2 + \norm{\vy_{K-1} - \vy^*_{K-1}}^2\right] + 2\E\left[\norm{\vx_{K-1}^* - \vx_{K}^*}^2 + \norm{\vy_{K-1}^* - \vy_{K}^*}^2\right] \\
         &\le 2\eps_{K-1} + \frac{144\beta \delta_{K-3}}{\mu_x \min\{\mu_x, \mu_y\}} ,
    \end{split}
    \end{align}}
    where we have used induction hypothesis (\ref{eq:induction-2}) and (\ref{eq:induction-6}). 
    
    Plugging (\ref{eq:inner-2}) into (\ref{eq:inner-1}), we have
    \begin{align}\label{eq:inner-3}\begin{split}
        &\quad \E \left[\norm{\vu_{K, T_{K}} - \vx_{K}^*}^2 + \norm{\vv_{K, T_{K}} - \vy_{K}^*} \right] \\
        &\le \frac{4}{R} \left( 2\eps_{K-1} + \frac{144\beta \delta_{K-3}}{\mu_x \min\{\mu_x, \mu_y\}}  \right) \\
        &= \frac{4}{R} \left( 2\eps_{K-1} + \frac{144\beta}{\mu_x \min\{\mu_x, \mu_y\}} \cdot \frac{8\Delta_f (1 - \rho)^{K-2}}{(\sqrt{q} - \rho)^2} \right) \\
        &= \frac{4}{R} \left(\frac{2}{1-\rho} + \frac{1728 \beta(L + \beta)(7(L+\beta) + 2\sqrt{n}\mu_y))}{\mu_x\mu_y \min\{\mu_x, \mu_y\} (1 - \rho)^2(\sqrt{q} - \rho)^2}\right) \eps_{K}
        \le \frac{1}{3}\eps_{K}.
    \end{split}\end{align}
    
    Consequently, combing inequality (\ref{eq:inner-3}) with equation (\ref{eq:prime-dual-distance}) in the proof of Lemma \ref{lem:SVRE-prime-dual}, we obtain
    \begin{align*}
        \E\left[\norm{\vx_{K} - \vx^*_{K}}^2 + \norm{\vy_{K} - \vy^*_{K}}^2\right] \le \frac{2(1 + \tau_K^2 (L+\beta)^2)}{3}\eps_{K} \le \frac{2(1 + 1/16)}{3}\eps_{K} \le \eps_{K},  
    \end{align*}
    and
    \begin{align*}
        &\quad \max_{\vy \in \fY} F_{K}(\vx_{K}, \vy) - F_{K}(\vx_K^*, \vy_{K}^*) \\
        &\le \left( \sqrt{2}(1 + \tau_K (L + \beta)) + 2(1 + \tau_K (L + \beta))^2 + 2 \right) \frac{(L+\beta)^2}{\mu_y} \frac{\eps_K}{3} + \frac{\eps_K}{6 \tau_K} \\
        &\le \left[ \left( \frac{5}{4}\sqrt{2} + \frac{50}{16} + 2 \right) \frac{(L+\beta)^2}{\mu_y} + 2\sqrt{n}(L+\beta) \right]\frac{\eps_{K}}{3} \\
        &\le \left( \frac{7(L + \beta)}{\mu_y} + 2\sqrt{n} \right)(L + \beta) \frac{\eps_K}{3} \\
        &= \frac{2\Delta_f}{9}(1 - \rho)^{K},
    \end{align*}
    where we have used $\tau_K = \frac{1}{4\sqrt{n}(L + \beta)}$ and
    \begin{align*}
        \begin{cases}
            \vx_{K} = \fP_{\fX} \left(\vu_{K, T_{K}} - \tau_{K} \nabla_{\vx} F_{K}(\vu_{K, T_{K}},  \vv_{K, T_{K}})\right) \\
            \vy_{K} = \fP_{\fY} \left(\vv_{K, T_{K}} + \tau_{K} \nabla_{\vy} F_{K}(\vu_{K, T_{K}},  \vv_{K, T_{K}})\right).
        \end{cases}
    \end{align*} 
    
    Therefore, we have proved inequalities (\ref{eq:induction-2}) and (\ref{eq:induction-3}) hold for $1 \le k \le K$.

    \paragraph{Part (b):  Inequality (\ref{eq:induction-4}) holds for $k = K$.} ~\\[0.1cm]
    Let $\Phi(\vx) \triangleq \max_{\vy \in \fY} f(\vx, \vy)$, and $\Phi_k(\vx) \triangleq \max_{\vy \in \fY} F_k(\vx, \vy) = \Phi(\vx) + \frac{\beta}{2}\norm{\vx - \vu_{k-1}}^2$. 
    It is easy to check that $\vx_k^* = \argmin_{\vx \in \fX} \Phi_k(\vx)$ and $\Phi_k$ is $\mu_x$-strongly convex by Lemma \ref{lem:phi-psi}. 
    
    
    Part (a) means we have (\ref{eq:induction-3}) holds for $1 \le k \le K$, which implies that
    \begin{align*}
        \E \left[\Phi_k(\vx_k)\right] - \Phi_k^* \le \frac{2\Delta_f}{9}(1 - \rho)^k, ~~~\text{ for } 1 \le k \le K.
    \end{align*}
    
    Consequently, according to the analysis of Catalyst for convex minimization 
    \cite{lin2018catalyst}\footnote{The original proof of Proposition 5 in Pages 45-46 of \cite{lin2018catalyst} should be slightly modified by employing $\Phi(\vx_0) - \Phi^* = \max_{\vy} f(\vx_0, \vy) - f(\vx^*, \vy^*) \le \Delta_f$. In fact, this result means inequality (\ref{eq:induction-4}) also holds for $k=0$.}, 
    it holds that
    \begin{align}\label{eq:induction-7}
        \E \Phi(\vx_k) - \Phi^* \le \frac{8\Delta_f}{(\sqrt{q} - \rho)^2} (1 - \rho)^{k+1} = \delta_k, ~~~\text{ for }~~ 0 \le k \le K.
    \end{align}
    
    Combing above inequality with the strong convexity of $\Phi$, we have
    \begin{align*}
        \frac{\mu_x}{2} \ \E\norm{\vx_k - \vx^*}^2 \le \E \Phi(\vx_k) - \Phi^* \le \delta_k, ~~~\text{ for }~~ 0 \le k \le K.
    \end{align*}
    
    
    \paragraph{Part (c): Inequality (\ref{eq:induction-6}) holds for $k = K$.}
    ~\\[0.1cm]
    For $K \ge 2$, the definition of $\vu_k$ and $\gamma < 1$ means
    \begin{align*}
        \norm{\vu_{K-1} - \vu_{K}} 
        &= \norm{\vx_{K-1} + \gamma(\vx_{K-1} - \vx_{K-2}) - \vx_{K} - \gamma(\vx_{K} - \vx_{K-1})} \\
        &\le (1 + \gamma)\norm{\vx_{K} - \vx_{K-1}} + \gamma \norm{\vx_{K-1} - \vx_{K-2}} \\
        &\le 3 \max\{\norm{\vx_{K} - \vx_{K-1}}, \norm{\vx_{K-1} - \vx_{K-2}}\}.
    \end{align*}
    Thus, we have
    {\small\begin{align*}
        \norm{\vu_{K-1} - \vu_{K}}^2 
        &\le 9 \max\{\norm{\vx_{K} - \vx_{K-1}}^2, \norm{\vx_{K-1} - \vx_{K-2}}^2\} \\
        &\le 9 \max\{2\norm{\vx_{K} - \vx^*}^2 + 2\norm{\vx_{K-1} - \vx^*}^2, 2\norm{\vx_{K-1} - \vx^*}^2 + 2\norm{\vx_{K-2} - \vx^*}^2\} \\
        &\le 36 \max\{\norm{\vx_{K} - \vx^*}^2, \norm{\vx_{K-1} - \vx^*}^2, \norm{\vx_{K-2} - \vx^*}^2\} \\
        &\le \frac{72 \delta_{K-2}}{\mu_x} ,
    \end{align*}}
    where the last inequality is according to the fact that (\ref{eq:induction-4}) holds for $0 \le k \le K$. 
    
    Next, by strong convexity of $F_K(\cdot, \vy_K^*)$ and $\vx_K^* = \argmin_{\vx \in \fX} F_K(\vx, \vy_K^*)$, we have
    \begin{align*}
        \frac{\mu_x + \beta}{2} \norm{\vx^*_K - \vx^*_{K+1}}^2 \le F_K(\vx_{K+1}^*, \vy_K^*) - F_K(\vx_K^*, \vy_K^*).
    \end{align*}
    Similarly, by strong concavity of $F_K(\vx_K^*, \cdot)$ and $\vy_K^* = \argmax_{\vy \in \fY} F_K(\vx_K^*, \vy)$, we have
    \begin{align*}
        \frac{\mu_y}{2} \norm{\vy^*_K - \vy^*_{K+1}}^2 \le F_K(\vx_{K}^*, \vy_K^*) - F_K(\vx_K^*, \vy_{K+1}^*).
    \end{align*}
    Therefore, we can conclude that
    \begin{align}\label{eq:star-star}
    \begin{split}
        &\quad \frac{\mu_x + \beta}{2} \norm{\vx^*_K - \vx^*_{K+1}}^2 + \frac{\mu_y}{2} \norm{\vy^*_K - \vy^*_{K+1}}^2 \\
        &\le F_K(\vx_{K+1}^*, \vy_K^*) - F_K(\vx_K^*, \vy_{K+1}^*) \\
        &= f(\vx_{K+1}^*, \vy_K^*) + \frac{\beta}{2} \norm{\vx_{K+1}^* - \vu_{K-1}}^2 - \left(f(\vx_{K}^*, \vy_{K+1}^*) + \frac{\beta}{2} \norm{\vx_{K}^* - \vu_{K-1}}^2\right) \\
        &= f(\vx_{K+1}^*, \vy_K^*) + \frac{\beta}{2} \left(\norm{\vx_{K+1}^* - \vu_{K}}^2 + \norm{\vu_K - \vu_{K-1}}^2 + 2\inner{\vx_{K+1}^* - \vu_{K}}{\vu_K - \vu_{K-1}} \right) \\
        &- \left(f(\vx_{K}^*, \vy_{K+1}^*) + \frac{\beta}{2} \left(\norm{\vx_{K}^* - \vu_{K}}^2 + \norm{\vu_K - \vu_{K-1}}^2 + 2\inner{\vx_{K}^* - \vu_{K}}{\vu_K - \vu_{K-1}}\right)\right) \\
        &= F_{K+1}(\vx_{K+1}^*, \vy_K^*) - F_{K+1}(\vx_K^*, \vy_{K+1}^*) + \beta \inner{\vx_{K+1}^* - \vx_{K}^*}{\vu_K - \vu_{K-1}} \\
        &\le \beta \inner{\vx_{K+1}^* - \vx_{K}^*}{\vu_K - \vu_{K-1}} 
        \le \frac{\beta}{2} \norm{\vx_{K+1}^* - \vx_{K}^*} + \frac{\beta}{2}\norm{\vu_K - \vu_{K-1}}^2,
    \end{split}
    \end{align}
    where we have used that $F_{K+1}(\vx_{K+1}^*, \vy_K^*) \le F_{K+1}(\vx_{K+1}^*, \vy_{K+1}^*) \le F_{K+1}(\vx_{K}^*, \vy_{K+1}^*)$.
    Hence, it holds that
    \begin{align*}
        \E\left[ \norm{\vx^*_K - \vx^*_{K+1}}^2 + \norm{\vy^*_K - \vy^*_{K+1}}^2 \right]
        \le \frac{\beta}{\min\{\mu_x, \mu_y\}}\E\norm{\vu_K - \vu_{K-1}}^2 \le \frac{72\beta \delta_{K-2}}{\mu_x \min\{\mu_x, \mu_y\}}.
    \end{align*}
    
    We also need to show the induction base to finish the proof of inequalities (\ref{eq:induction-2}) to (\ref{eq:induction-6}).
    
    \paragraph{Part (d): Induction base.} ~\\[0.1cm]
    Finally, we present the induction base that inequalities (\ref{eq:induction-2}) to (\ref{eq:induction-6}) hold for $k=1$.
    
    Since $F_1$ is $(\mu_x + \beta, \mu_y)$-convex-concave, we have
    \begin{align*}
        &\quad \frac{\mu_x + \beta}{2} \norm{\vx_0 - \vx_1^*}^2  + \frac{\mu_y}{2} \norm{\vy_0 - \vy_1^*}^2 \\
        &\le F_1(\vx_0, \vy_1^*) - F_1(\vx_1^*, \vy_1^*) + F_1(\vx_1^*, \vy_1^*) - F_1(\vx_1^*, \vy_0) \\
        &= f(\vx_0, \vy_1^*) - f(\vx_1^*, \vy_0) - \frac{\beta}{2} \norm{\vx_1^* - \vx_0}^2 \\
        &\le \max_{\vy \in \fY} f(\vx_0, \vy) - \min_{\vx \in \fX} f(\vx, \vy_0) = \Delta_f,
    \end{align*}
    which implies
    \begin{align}\label{ieq:induction-help}
    \begin{split}
        & \E \left[\norm{\vu_{1, T_1} - \vx_{1}^*}^2 + \norm{\vv_{1, T_1} - \vy_{1}^*} \right] \\
        \le & 4 e^{-\theta T_1} \E \left[\norm{\vx_{0} - \vx_{1}^*}^2 + \norm{\vy_{0} - \vy_{1}^*}^2\right] \\
        \le & \frac{4}{R} \frac{2\Delta_f}{\min \{\mu_x, \mu_y\}} \le \frac{1}{3}\eps_{1}.
    \end{split}
    \end{align}
    
    Similar to the proof in Part (a), we can use Lemma \ref{lem:SVRE-prime-dual}, \ref{lem:average-smooth} and inequality (\ref{ieq:induction-help}) to show that (\ref{eq:induction-2}) and (\ref{eq:induction-3}) hold for $k=1$.
    
    Moreover, according to Proposition 5 in \cite{lin2018catalyst} and the strong convexity of $F_1$, we know that (\ref{eq:induction-4}) holds for $k=1$. 
    
    Note that $\vz_1 - \vz_0 = (1 + \gamma)(\vy_1 - \vy_0)$, thus we have
    \begin{align*}
        \E \norm{\vu_0 - \vu_1}^2 \le 4 \E \norm{\vx_1 - \vx_0}^2 \le \frac{32\delta_0}{\mu_x} \le \frac{72 \delta_{-1}}{\mu_x}.
    \end{align*}    
    Therefore, by inequality (\ref{eq:star-star}) in Part (c), we have
    \begin{align*}
        \E\left[ \norm{\vx_1^* - \vx_2^*}^2 + \norm{\vy_1^* - \vy_2^*}^2 \right] \le \frac{\beta}{\min\{\mu_x, \mu_y\}}\E\norm{\vu_1 - \vu_0}^2 \le \frac{72\beta\delta_{-1}}{\mu_x \min\{\mu_x, \mu_y\}}.
    \end{align*}
    
    As a conclusion, we have proved inequalities (\ref{eq:induction-2}) to (\ref{eq:induction-6}) hold for any $k \ge 1$. \\[0.1cm]
\end{proof}

\subsection{The proof of Theorem~\ref{thm:asvre}}
\begin{proof}
    Consider that the fact
    \begin{align*}
        \vy_k^* = \argmax_{\vy \in \fY} \left(f(\vx_k^*, \vy) + \frac{\beta}{2}\norm{\vx_k^* - \vu_{k-1}}^2\right) = \argmax_{\vy \in \fY} f(\vx_k^*, \vy) = \vy_f^*(\vx_k^*),
    \end{align*} 
    then we have
    \begin{align}\label{eq:xk-star-x-star}
        \norm{\vy_k^* - \vy^*} = \norm{\vy_f^*(\vx_k^*) - \vy_f^*(\vx^*)} \le \kappa_y \norm{\vx_k^* - \vx^*}
    \end{align}
    by using Lemma \ref{lem:phi-psi}.
    
    Therefore, we conclude that
    \begin{align*}
        &\quad \E \left[\norm{\vu_{k, T_k} - \vx^*}^2 + \norm{\vv_{k, T_k} - \vy^*}^2\right] \\
        &\le 2\E \left[\norm{\vu_{k, T_k} - \vx_k^*}^2 + \norm{\vv_{k, T_k} - \vy_k^*}^2\right] + 2\E\left[ \norm{\vx_k^* - \vx^*}^2 + \norm{\vy_k^* - \vy^*}^2 \right] \\
        &\le \frac{2}{3} \eps_k + 2(\kappa_y^2 + 1) \E \norm{\vx_k^* - \vx^*}^2 \\
        &\le \eps_k + 4(\kappa_y^2 + 1) \E \left[\norm{\vx_k^* - \vx_k}^2 + \norm{\vx_k - \vx^*}^2\right] \\
        &\le \eps_k + 4(\kappa_y^2 + 1) \left( \eps_k + \frac{2\delta_k}{\mu_x} \right) \triangleq \hat\eps,
    \end{align*}
    where the second inequality is due to (\ref{eq:inner-3}) and (\ref{eq:xk-star-x-star}) and the last one is according to (\ref{eq:induction-2}) and (\ref{eq:induction-4}).
    
    Note that
    \begin{align*}
        \vy_k = \fP_{\fY} \left(\vv_{k, T_k} + \tau_k \nabla_{\vy} F_k(\vu_{k, T_k}, \vv_{k, T_k})\right) = \fP_{\fY} \left(\vv_{k, T_k} + \tau_k \nabla_{\vy} f(\vu_{k, T_k}, \vv_{k, T_k})\right).
    \end{align*}
    Then, Lemma \ref{lem:SVRE-prime-dual} and \ref{lem:average-smooth} implies
    \begin{align*}
        &\quad \E\left[f(\vx^*, \vy^*) - \min_{\vx \in \fX} g(\vx, \vy_k)\right]  \\
        \le & \left( \sqrt{2}\left(1 + \frac{L}{4\sqrt{n}(L + \beta)}\right) + 2\left(1 + \frac{L}{4\sqrt{n}(L + \beta)}\right)^2 + 2 \right) \kappa_x L \hat\eps + 2\sqrt{n}(L + \beta) \hat\eps \\
        &\le 7 \kappa_x L \hat\eps + 2\sqrt{n}(L + \beta)\hat\eps 
        \le \left( 7 \kappa_x L + 2\sqrt{n}(L + \beta) \right) \left( (4\kappa_y^2 + 5)\eps_k + \frac{8(\kappa_y^2 + 1)}{\mu_x} \delta_k \right) \\
        &\le \left( 7 \kappa_x L + 2\sqrt{n}(L + \beta) \right) \left( \frac{2 \Delta_f(4\kappa_y^2 + 5) \mu_y}{3(L + \beta)(7(L + \beta) + 2\sqrt{n} \mu_y) } + \frac{64\Delta_f(\kappa_y^2 + 1)(1 - \rho)}{\mu_x(\sqrt{q} - \rho)^2} \right) (1 - \rho)^{k} \\
        &\le \left( \frac{2 (4 \kappa_y^2 + 5)(\mu_x + \mu_y)}{3 \mu_x} + \frac{128 \kappa_y^2 \left( 7 \kappa_x L + 2\sqrt{n}(L + \beta) \right) }{\mu_x (\sqrt{q} - \rho)^2} \right) \Delta_f (1 - \rho)^{k}
    \end{align*}
    Together with inequality (\ref{eq:induction-7}), we obtain that
    \begin{align*}
        &\quad \E\left[\max_{\vy \in \fY} f(\vx_k, \vy) - \min_{\vx \in \fX} f(\vx, \vy_k)\right] \\
        &\le \Delta_f (1 - \rho)^{k} \left( \frac{2 (4 \kappa_y^2 + 5)(\mu_x + \mu_y)}{3 \mu_x} + \frac{128 \kappa_y^2 \left( 7 \kappa_x L + 2\sqrt{n}(L + \beta) \right) }{\mu_x (\sqrt{q} - \rho)^2} + \frac{8}{(\sqrt{q} - \rho)^2} \right) \\
        &\le \Delta_f (1 - \rho)^{k} \left( 12\kappa_y^2 \kappa_x + \frac{128 \kappa_y^2 \left( 7 \kappa_x L + 2\sqrt{n}(L + \beta) \right) }{\mu_x (\sqrt{q} - \rho)^2} + \frac{8}{(\sqrt{q} - \rho)^2} \right) \\
        &\le \frac{916\Delta_f\left( \kappa_x L + \sqrt{n}(L + \beta) \right)\kappa_y^2}{\mu_x(\sqrt{q} - \rho)^2} (1 - \rho)^{k}.
    \end{align*}
\end{proof}

\subsection{The Proof of Corollary \ref{cor:ASVRE}}
\begin{proof}
    First, note that $\beta = \min\{\mu_y - \mu_x, \max\{L / \sqrt{n} - \mu_x, 0\}\} \le \mu_y - \mu_x \le L$
    and
    \begin{align*}
        \frac{1}{q} = \frac{\mu_x + \beta}{\mu_x} \le \frac{\mu_y}{\mu_x}.
    \end{align*}
    Together with Theorem \ref{thm:asvre}, we have
    \begin{align*}
        &\quad \E \left[\max_{\vy \in \fY} f(\vx_K, \vy) - \min_{\vx\in\fX} f(\vx, \vy_K)\right] 
        \le e^{-\rho K} \frac{916\Delta_f\left( \kappa_x L + \sqrt{n}(L + \beta) \right)\kappa_y^2}{\mu_x(\sqrt{q} - \rho)^2} \\
        &\le e^{-\rho K} \left( \frac{2748\sqrt{n}\Delta_f \kappa_y^2 \kappa_x^2}{(\sqrt{q} - \rho)^2} \right) 
        \le e^{-\rho K} \left( 10992\sqrt{n}\Delta_f \kappa_y \kappa_x^3 \right) \le \eps,
    \end{align*}
    where we use that $\rho = 0.5 \sqrt{q}$. 
    
    Recall that
    \begin{align*}
        & T_k \\
        =& \ceil{4\left(n + \frac{2 \sqrt{n}(L+\beta)}{\min\{\mu_x + \beta, \mu_y\}}\right) \log\left( 12 \left(\frac{2}{1-\rho} + \frac{1728 \beta(L + \beta)(7(L+\beta) + 2\sqrt{n}\mu_y))}{\mu_x\mu_y \min\{\mu_x, \mu_y\} (1 - \rho)^2(\sqrt{q} - \rho)^2}\right) \right)} \\
        \le& \ceil{4\left(n + \frac{4 \sqrt{n}L}{\min\{\mu_x + \beta, \mu_y\}}\right) \log\left( 12 \left(4 + \frac{884736 \sqrt{n} L^3 )}{\mu_x^2 \mu_y q}\right) \right)} \\
        \le& \ceil{4\left(n + \frac{4 \sqrt{n}L}{\min\{\mu_x + \beta, \mu_y\}}\right) \log\left( 10616880  \sqrt{n} \kappa_x^3 \right)},
    \end{align*}
    where we have noticed that $\rho \le 0.5$.
    
    Therefore, the total SFO complexity of AL-SVRE is 
    \begin{align*}
        \sum_{k=1}^K (T_k + n) = \fO\left(\left(\frac{2 n}{\sqrt{q}} + \frac{2}{\sqrt{q}} \left( n + \frac{4\sqrt{n}L}{\min\{\mu_x + \beta, \mu_y\}} \right)\log(\sqrt{n}\kappa_x^3)\right) \log(\sqrt{n}\Delta_f \kappa_y \kappa_x^3/\eps) \right).
    \end{align*}
    
    Observe that
\begin{enumerate}
    \item if $\kappa_y \ge \sqrt{n}$, we have $\mu_y \le L/\sqrt{n}$, $\beta = \mu_y - \mu_x$, which means
    \begin{align*}
        \frac{n}{\sqrt{q}} &= n \sqrt{\frac{\mu_y}{\mu_x}} \le n^{3/4} \sqrt{\kappa_x}
    \end{align*}
    and
    \begin{align*}    
        \frac{1}{\sqrt{q}} \cdot \frac{\sqrt{n} L}{\min\{\mu_x + \beta, \mu_y\}} &=  \frac{\sqrt{n} L}{\mu_y} \sqrt{\frac{\mu_y}{\mu_x}} = \sqrt{n} \sqrt{\kappa_x \kappa_y};
    \end{align*}
    \item if $\kappa_x > \sqrt{n} > \kappa_y$, we have $\beta = L/\sqrt{n} - \mu_x$, $\mu_x + \beta = L/\sqrt{n} < \mu_y$ , which means
    \begin{align*}
        \frac{n}{\sqrt{q}} &= n \sqrt{\frac{L}{\sqrt{n}\mu_x}} = n^{3/4} \sqrt{\kappa_x} 
    \end{align*}
    and
    \begin{align*}  
        \frac{1}{\sqrt{q}} \cdot \frac{\sqrt{n} L}{\min\{\mu_x + \beta, \mu_y\}} &= \frac{\sqrt{n} L}{L/\sqrt{n}} \sqrt{\frac{L}{\sqrt{n}\mu_x}} = n^{3/4} \sqrt{\kappa_x};
    \end{align*}
    \item if $\sqrt{n} \ge \kappa_x$, we have $\beta = 0$, which means
    \begin{align*}
        \frac{n}{\sqrt{q}} = n \quad \text{and} \quad
        \frac{1}{\sqrt{q}} \cdot \frac{\sqrt{n} L}{\min\{\mu_x + \beta, \mu_y\}} = \sqrt{n} \kappa_x \le n.
    \end{align*}
\end{enumerate}

Hence, we can conclude that the total SFO complexity of AL-SVRE is \begin{align*}
    \sum_{k=1}^K (T_k + n) = \fO\left( \sqrt{n} \sqrt{(\sqrt{n} + \kappa_x)(\sqrt{n} + \kappa_y)}\log(\sqrt{n}\kappa_x^3) \log(\sqrt{n}\Delta_f \kappa_y \kappa_x^3/\eps) \right).
\end{align*}
    
\end{proof}

%% file: appendix/lower.tex
\section{The Proof of Theorem \ref{thm:balance-lower-bound}}\label{appendix:lower}

The construction for the lower bound in Theorem \ref{thm:balance-lower-bound} follows from the idea of ``zero-chain'' property~\cite{zhang2019lower,han2021lower}. Since we focus on SFO algorithm without proximal operator, our analysis is simpler than \citet{han2021lower} and \citet{xie2020lower}'s.

Without loss of generality, we assume that the SFO algorithm starts iteration at $(\vx^{(0)}, \vy^{(0)}) = (\vzero_{d_x}, \vzero_{d_y})$. Otherwise, we can take the objective function
{\small\begin{align*}
    \hat{f}(\vx, \vy) = \frac{1}{n} \sum_{i=1}^n f_i(\vx + \vx^{(0)}, \vy + \vy^{(0)})
\end{align*}}
into consideration.

Consider following function $H: \BR^d \times \BR^d \to \BR$ defined as 
\begin{align}\label{eq:H}
    H(\vx, \vy; \alpha, d) = \frac{\alpha}{2} \norm{\vx}^2 + \vx^{\top} (\mB \vy - \vc) - \frac{\alpha}{2} \norm{\vy}^2,
\end{align}
where
\begin{align*}
    \mB = \begin{bmatrix}
        1  &   & & & \\
        -1 & 1 & & & \\
           & \ddots & \ddots & & \\
           & & -1 & 1 & \\
          & & & -1 & \sqrt{\alpha\omega}
    \end{bmatrix} \in \BR^{d \times d},
\end{align*}
$\vc = \left( \omega, 0, 0, \dots, 0\right)^{\top}$ and  $\omega = \frac{\sqrt{\alpha^2 + 4} - \alpha}{2}$.

Furthermore, we define subspaces as follows
    \begin{align*}
    \fF_k = \begin{cases}
    \spn\{ \ve_1, \ve_2, \dots, \ve_k \}, & k=1,\dots,d, \\[0.1cm]
    \{\vzero_{d}\}, & k=0,
    \end{cases}
    \end{align*}
    where $\{\ve_1,\dots,\ve_{d}\}$ is the standard basis of $\BR^{d}$.

Here, we state some properties of the Function $H$ in the above definition.
\begin{lem}\label{prop:H}
    For the function $H$ defined in Equation (\ref{eq:H}), following properties hold.
\begin{enumerate}
    \item $H$ is $\sqrt{8 + 2 \alpha^2}$-smooth.
    \item The saddle point of function $H$ is
    \begin{align*}
        \begin{cases}
            \vx^* = (q, q^2, \dots, q^d)^{\top}, \\
            \vy^* = \omega \left( q, q^2, \dots, q^{d-1}, \frac{1}{\sqrt{1 - q}} q^d \right)^{\top},
        \end{cases}
    \end{align*}
    where $q = \frac{2 + \alpha^2 - \alpha \sqrt{\alpha^2 + 4}}{2}$.
    \item For $k < d$, if $(\vx, \vy) \in \fF_k \times \fF_k$, then
    $
        (\nabla_{\vx} H(\vx, \vy), \nabla_{\vy} H(\vx, \vy)) \in \fF_{k+1} \times \fF_{k+1}.
    $
    \item For $k \le d/2$ and $(\vx, \vy) \in \fF_k \times \fF_k$, we have
    \begin{align*}
        \frac{\norm{\vx - \vx^*}^2 + \norm{\vy - \vy^*}^2}{\norm{\vx^*}^2 + \norm{\vy^*}^2} \ge \frac{1}{2} q^{2k}.
    \end{align*}
\end{enumerate}
\end{lem}

\begin{proof}
    1. Note that $\alpha \omega = \frac{2 \alpha}{\sqrt{\alpha^2 + 4} + \alpha} < 1$. Hence $\norm{\mB} \le \Norm{\mB}_1 = 2$. \\
    Then for any $(\vx_1, \vy_1), (\vx_2, \vy_2) \in \BR^d \times \BR^d$, we have
    \begin{align*}
        &\quad \norm{\begin{bmatrix} \nabla_{\vx} H(\vx_1, \vy_1) - \nabla_{\vx} H(\vx_2, \vy_2) \\ \nabla_{\vy} H(\vx_1, \vy_1) - \nabla_{\vy} H(\vx_2, \vy_2) \end{bmatrix}}^2
        = \norm{\begin{bmatrix} \alpha (\vx_1 - \vx_2) + \mB (\vy_1 - \vy_2) \\ \mB^{\top} (\vx_1 - \vx_2) - \alpha (\vy_1 - \vy_2) \end{bmatrix}}^2 \\
        &\le 2\alpha^2 \norm{\vx_1 - \vx_2}^2 + 2 \norm{\mB}^2 \norm{\vy_1 - \vy_2}^2 + 2 \norm{\mB}^2 \norm{\vx_1 - \vx_2}^2 + 2 \alpha^2 \norm{\vy_1 - \vy_2}^2 \\
        &\le (8 + 2 \alpha^2) \left( \norm{\vx_1 - \vx_2}^2 + \norm{\vy_1 - \vy_2}^2 \right),
    \end{align*}
    where the first inequality is according to $\norm{\va + \vb}^2 \le 2 \norm{\va}^2 + 2 \norm{\vb}^2$.
    
    2. Letting the gradient of $H$ equal to zero, we get that
    \begin{align*}
        \nabla_{\vx} H(\vx, \vy) &= \alpha \vx + \mB \vy - \vc = 0, \\
        \nabla_{\vy} H(\vx, \vy) &= \mB^{\top} \vx - \alpha \vy = 0. 
    \end{align*}
    Hence, the saddle point of $H$ satisfies 
    \begin{align}
        \vy^* = \frac{1}{\alpha} \mB^{\top} \vx^*, \\
        \left( \alpha^2 I + \mB \mB^{\top} \right) \vx^* = \alpha \vc. \label{eq:x-star}
    \end{align}
    Equation (\ref{eq:x-star}) are equivalent to
    \begin{align*}
        \begin{bmatrix}
            1 + \alpha^2 & -1 & & &  \\
            -1 & 2 + \alpha^2 & -1 & &  \\
            & \ddots & \ddots & \ddots & \\
            & & -1 & 2 + \alpha^2 & -1 \\
            & & & -1 & 1 + \alpha \omega + \alpha^2
        \end{bmatrix}
        \vx^*
        = \begin{bmatrix}
            \alpha \omega \\
            0 \\
            \vdots \\
            0 \\
            0
        \end{bmatrix}.
    \end{align*}
    
    Note that $q$ is a root of the equation $s^2 - (2 + \alpha^2)s + 1 = 0$, then we have
    \begin{align*}
        (1 + \alpha^2) q - q^2 &= 1 - q = \frac{\alpha \sqrt{\alpha^2 + 4} - \alpha^2}{2} = \alpha \omega, \\
        -q^{d-1} + (1 + \alpha \omega + \alpha^2 )q^d &= q^{d-1} \left( -1 + (2 + \alpha^2) q - q + \alpha \omega q \right) \\
        &= q^{d-1} \left( -1 + (2 + \alpha^2) q - q + (1 - q) q \right) = 0,
    \end{align*}
    and 
    \begin{align*}
        -q^k + (2 + \alpha^2) q^{k+1} -q^{k+2} &= 0, ~~~\text{ for } k = 1, 2, \dots, d-2, 
    \end{align*}
    which implies the solution to Equation (\ref{eq:x-star}) is $\vx^* = (q, q^2, \dots, q^d)^{\top}$.
    Additionally, we have
    \begin{align*}
        \vy^* = \frac{1}{\alpha} \mB^{\top} \vx^* = \omega \left( q, q^2, \dots, q^{d-1}, \frac{1}{\sqrt{1 - q}} q^d \right)^{\top},
    \end{align*}
    where we have used that $\alpha \omega = 1 - q$.
    
    3. For $(\vx, \vy) \in \fF_k$ for $k < d$, note that
    \begin{align*}
        \vc \in \fF_{1} \subseteq \fF_{k+1}, ~~ \mB \vy \in \fF_{k+1}, ~~
        \mB^{\top} \vx \in \fF_{k+1}.
    \end{align*}
    Therefore we have $\nabla_{\vx} H(\vx, \vy), \nabla_{\vy} H(\vx, \vy) \in \fF_{k+1}$.
    
    4. For $(\vx, \vy) \in \fF_k$, we have $x_{k+1} = \dots = x_d = y_{k+1} = \dots = y_d = 0$ and 
    \begin{align*}
        \norm{\vx - \vx^*}^2 + \norm{\vy - \vy^*}^2 
        &\ge \sum_{i = k+1}^d q^{2i} + \omega^2 \sum_{i=k+1}^{d-1} q^{2i} + \frac{\omega^2}{1 - q} q^{2d} \\
        &= (1 + \omega^2) \frac{q^{2(k+1)} (1 - q^{2(d - k)})}{1 - q^2} + \frac{\omega^2 q}{1 - q} q^{2d}.
    \end{align*}
    
    Consequently, there holds
    \begin{align*}
        \frac{\norm{\vx - \vx^*}^2 + \norm{\vy - \vy^*}^2}{\norm{\vx^*}^2 + \norm{\vy^*}^2} 
        &\ge \frac{(1 + \omega^2) \frac{q^{2(k+1)} (1 - q^{2(d - k)})}{1 - q^2} + \frac{\omega^2 q}{1 - q} q^{2d}}{(1 + \omega^2) \frac{q^{2} (1 - q^{2d})}{1 - q^2} + \frac{\omega^2 q}{1 - q} q^{2d}} \\
        &\ge \frac{q^{2k} (1 - q^{2(d-k)})}{1 - q^{2d}} \\
        &\ge \frac{q^{2k} (1 - q^{2(d-k)})}{1 - (2 q^d - 1)} \\
        &= \frac{1}{2} q^{2k} \frac{1 - q^{2(d - k)}}{1 - q^d} \ge \frac{1}{2} q^{2k},
    \end{align*}
    where we have used that $1 + q^{2d} \ge 2 q^d$ and $2(d - k) \ge d$ according to $k \le d/2$.
\end{proof}
    
With these pieces in hand, we now construct our adversary problem as follows:
\begin{align}\label{def:scsc}
    \min_{\vx} \max_{\vy} f(\vx, \vy; \alpha, \lambda, d) \triangleq \frac{1}{n} \sum_{i=1}^n f_i(\vx, \vy; \alpha, \lambda, d) = \frac{1}{n} \sum_{i=1}^n \lambda H(\mU_i \vx, \mU_i \vy),
\end{align}
where $f: \BR^{n d} \times \BR^{n d} \to \BR$, and matrices $\mU_1,\dots,\mU_n \in \BR^{d \times n d}$ consist a partition of the identity matrix of order $n d$ such that $\mI=[\mU_1^{\top}, \cdots, \mU_n^{\top}]$.
The trick of constructing worse objective function for SFO algorithms by matrices $\mU_1^{\top}, \cdots, \mU_n^{\top}$ also can be found in the analysis for minimization problem~\cite{lan2017optimal, zhou2019lower}.

The smoothness, convexity and concavity of $f$ can be characterized as follows.
\begin{lem}\label{lem:lower-smooth}
    The class of functions $\{f_i\}$ defined in equation (\ref{def:scsc}) is $\lambda \sqrt{\frac{8 + 2\alpha^2}{n}}$-average-smooth; and the function $f$ is $(\frac{\lambda \alpha}{n}, \frac{\lambda \alpha}{n})$-convex-concave. 
\end{lem}
\begin{proof}
    Using the fact $\sum_{i=1}^n \norm{\mU_i \vx}^2 = \norm{\vx}^2$, we have
    \begin{align*}
        f(\vx, \vy) = \frac{\lambda \alpha}{n} \norm{\vx}^2 - \frac{\lambda}{n} \vx^{\top} \left( \sum_{i=1}^n \mU_i^{\top} \vc \right) + \frac{\lambda}{n} \vx^{\top} \left( \sum_{i=1}^n \mU_i^{\top} \mB \mU_i \right) \vy - \frac{\lambda \alpha}{n} \norm{\vy}^2.
    \end{align*}
    It is clear that $f$ is $(\frac{\lambda \alpha}{n}, \frac{\lambda \alpha}{n})$-convex-concave. 
    
    Moreover, for any $(\vx_1, \vy_1), (\vx_2, \vy_2) \in \BR^{nd} \times \BR^{nd}$, it holds that
    \begin{align*}
        &\quad \frac{1}{n} \sum_{i=1}^n \left( \norm{\nabla_{\vx} f_i(\vx_1, \vy_1) - \nabla_{\vx} f_i(\vx_2, \vy_2)}^2 + \norm{\nabla_{\vy} f_i(\vx_1, \vy_1) - \nabla_{\vy} f_i(\vx_2, \vy_2)}^2 \right) \\
        &= \frac{\lambda^2}{n} \sum_{i=1}^n \bigg( \norm{ \mU_i^{\top} \left(\nabla_{\vx} H(\mU_i\vx_1, \mU_i \vy_1) - \nabla_{\vx} H(\mU_i\vx_2, \mU_i \vy_2)\right)}^2 \\ 
        &\quad\quad\quad\quad\quad+ \norm{ \mU_i^{\top} \left(\nabla_{\vy} H(\mU_i\vx_1, \mU_i \vy_1) - \nabla_{\vy} H(\mU_i\vx_2, \mU_i \vy_2)\right)}^2\bigg) \\
        &= \frac{\lambda^2}{n} \sum_{i=1}^n \bigg( \norm{ \nabla_{\vx} H(\mU_i\vx_1, \mU_i \vy_1) - \nabla_{\vx} H(\mU_i\vx_2, \mU_i \vy_2)}^2 + \norm{ \nabla_{\vy} H(\mU_i\vx_1, \mU_i \vy_1) - \nabla_{\vy} H(\mU_i\vx_2, \mU_i \vy_2)}^2\bigg) \\
        &\le \frac{\lambda^2 (8 + 2\alpha^2)}{n} \sum_{i=1}^n \bigg( \norm{\mU_i(\vx_1 - \vx_2)}^2 + \norm{\mU_i(\vy_1 - \vy_2)}^2 \bigg) \\
        &= \frac{\lambda^2 (8 + 2\alpha^2)}{n} \left( \norm{\vx_1 - \vx_2}^2 + \norm{\vy_1 - \vy_2}^2 \right),
    \end{align*}
    where the first inequality follows from smoothness of $H$ (cf. Property 1 in Lemma \ref{prop:H}). 
\end{proof}

Each component function $f_i$ has ``zero-chain'' property for stochastic first-order oracle. That is, the information provided by an SFO call at the current point $(\vx, \vy)$ can at most increase the dimension of the linear space which contains $(\vx, \vy)$ by 1. We present the formal statement in Lemma \ref{lem:zero-chain}.

\begin{lem}\label{lem:zero-chain}
    Let $(\vx^{(t)}, \vy^{(t)})$ be the point obtained by an SFO algorithm $\fA$ at time-step $t$ and denote $k^{(t)}_{i} \triangleq |\{s \le t: i_s = i\}|$. Then there holds
    \begin{align}\label{eq:zero-chain}
        \mU_i \vx^{(t)}, \mU_i \vy^{(t)} \in \fF_{k^{(t)}_{i}}, ~~~\text{ for } t \ge 0, i = 1, \dots, n.
    \end{align}
\end{lem}

\begin{proof}
    For $t = 0$, Equation (\ref{eq:zero-chain}) holds apparently by $\vx^{(0)} = \vy^{(0)} = \vzero$. 
    
    Now we suppose Equation (\ref{eq:zero-chain}) holds for $t < T$. It is easily to check that $k_{i}^{(T)} = k_{i}^{(T -1)}$ for $i \neq i_T$ and $k_{i_T}^{(T)} = k_{i_T}^{(T -1)} + 1$.
    
    Observe that for $i \neq i_T$ we have
    \begin{align*}
        \mU_i \nabla_{\vx} f_{i_T}(\vx, \vy) = \lambda \mU_i \mU_{i_T}^{\top} \nabla_{\vx} H(\mU_{i_T} \vx, \mU_{i_T} \vy) = \vzero, 
    \end{align*}
    which implies
    \begin{align*}
        \mU_i \vx^{(T)} &\in \spn\big\{\mU_i \vx^{(0)}, \dots, \mU_i \vx^{(T-1)}, \mU_i \nabla_{\vx} f_{i_T}(\vx^{(0)}, \vy^{(0)}), \dots, \mU_i \nabla_{\vx} f_{i_T}(\vx^{(T-1)}, \vy^{(T-1)})\big\} \\
        &= \spn\big\{\mU_i \vx^{(0)}, \dots, \mU_i \vx^{(T-1)}\big\} \subseteq \fF_{k_i^{(T-1)}} = \fF_{k_i^{(T)}}.
    \end{align*}
    Next, Following from Property 3 in Lemma \ref{prop:H}, we know that for $t \le T-1$ 
    \begin{align*}
        \mU_{i_T} \nabla_{\vx} f_{i_T}(\vx^{(t)}, \vy^{(t)}) &= \lambda \mU_{i_T} \mU_{i_T}^{\top} \nabla_{\vx} H(\mU_{i_T} \vx^{(t)}, \mU_{i_T} \vy^{(t)}) \\
        &= \lambda \nabla_{\vx} H(\mU_{i_T} \vx^{(t)}, \mU_{i_T} \vy^{(t)}) \in \fF_{k_{i_T}^{(T-1)} + 1} = \fF_{k_{i_T}^{(T)}}, 
    \end{align*}
    where we have used the inductive hypothesis that
    \begin{align*}
        \mU_{i_T} \vx^{(t)}, \mU_{i_T} \vy^{(t)} \in \fF_{k_{i_T}^{(t)}} \subseteq \fF_{k_{i_T}^{(T-1)}}.
    \end{align*}
    Therefore, we can conclude that 
    \begin{align*}
        & \mU_{i_T} \vx^{(T)} \\
        \in & \spn\big\{\mU_{i_T} \vx^{(0)}, \dots, \mU_{i_T} \vx^{(T-1)}, \mU_{i_T} \nabla_{\vx} f_{i_T}(\vx^{(0)}, \vy^{(0)}), \dots, \mU_{i_T} \nabla_{\vx} f_{i_T}(\vx^{(T-1)}, \vy^{(T-1)})\big\} \\
        \subseteq & \fF_{k_{i_T}^{(T)}}.
    \end{align*}
    The reason for the of result $\mU_i \vx^{(T)} \in \fF_{k_i^{(T)}}$ is same. 
\end{proof}


Now we can show the lower bound for finding an approximate saddle point of problem~(\ref{def:scsc}) by SFO algorithms when $L/\mu = \Omega(\sqrt{n})$. 

\begin{thm}\label{thm:lower}
    For the parameter $L, \mu, n, \eps$ such that $L/\mu > \sqrt{10 n}$ and $\eps < \frac{1}{2} e^{-5} \approx 0.00337$, we set 
    \begin{align*}
        \alpha = \sqrt{\frac{8n}{L^2/\mu^2 - 2 n}}, ~~~\lambda = \frac{n \mu}{\alpha}, ~~~ d = \floor{\frac{1}{ \alpha} \ln \left(\frac{1}{2\eps}\right)} - 4.
    \end{align*}
    Then the functions $f(\vx, \vy; \alpha, \lambda, d)$ and $f_i(\vx, \vy; \alpha, \lambda, d)$ defined in Problem (\ref{def:scsc}) satisfy $f$ is $(\mu, \mu)$-convex-concave and $\{f_i\}_{i=1}^n$ is $L$-average-smooth. Moreover, when we employ any SFO algorithm $\fA$ to solve the Problem (\ref{def:scsc}), there holds
    \begin{align*}
        \E \left[ \norm{\vx^{(t)} - \vx^*}^2 + \norm{\vy^{(t)} - \vy^*}^2 \right] > \eps ~~~ \text{ for } t \le n d / 2.
    \end{align*}
\end{thm}

\begin{proof}
    By Lemma \ref{lem:lower-smooth} and definition of $\alpha, \lambda$, it is clear that $f$ is $(\mu, \mu)$-convex-concave and $\{f_i(\vx, \vy; \alpha, \lambda, d)\}_{i=1}^n$ is $L$-average-smooth. 
    
    For $T = n d/2$, let $i = \argmin_{j} \{ k_{j}^{(T)} \}$. 
    It is clear that $k_i^{(T)} \le d/2$. 
    Then by Property 4 in Lemma \ref{prop:H}, for $t \le T$, we have
    \begin{align*}
        \E \left[\norm{\vx^{(t)} - \vx^*}^2 + \norm{\vy^{(t)} - \vy^*}^2\right] 
        &\ge \E \left[\norm{ \mU_i (\vx^{(t)} - \vx^*)}^2 + \norm{ \mU_i (\vy^{(t)} - \vy^*)}^2\right] \\
        &\ge \frac{q^{k_i^{(t)}}}{2} \left( \norm{ \mU_i \vx^*}^2 + \norm{ \mU_i \vy^*}^2 \right) \\
        &\ge \frac{q^{d/2}}{2} q^2 \frac{1 - q^{2d}}{1 - q^2} \\
        &\ge \frac{q^{d/2 + 2}}{2},
    \end{align*}
    where $q = \frac{2 + \alpha^2 - \alpha \sqrt{\alpha^2 + 4}}{2}$, $\omega = \frac{\sqrt{\alpha^2 + 4} - \alpha}{2}$, and $\mU_i \vx^* = (q, q^2, \dots, q^d)^{\top}$ by Property 2 in Lemma \ref{prop:H}. 
    
    Next, note that 
    \begin{align*}
        (2 + d/2) \ln(1/q) &= (2 + d/2) \ln \left( 1 + \frac{\alpha(\alpha + \sqrt{\alpha^2 + 4})}{2} \right) \\
        &\le (2 + d/2) \frac{\alpha(\alpha + \sqrt{\alpha^2 + 4})}{2} \\
        &< (2 + d/2) \alpha (\alpha + 1) \\
        &\le (4 + d) \alpha \\ 
        &\le \ln(1/2\eps)
    \end{align*}
    which means that $\frac{q^{d/2 + 2}}{2} > \eps$. 
    The first inequality is according to $\ln(1 + a) \le a$, the second inequality follows from $\sqrt{a + b} < \sqrt{a} + \sqrt{b}$, the third inequality is due to $L/\mu \ge \sqrt{10 n}$ and $\alpha = \sqrt{\frac{8n}{L^2/\mu^2 - 2n}} \le 1$, and the last inequality is based on the definition of $d$. 
\end{proof}

We remark that the condition $\eps < \frac{1}{2} e^{-5}$ can ensure that $d \ge 1$.

Furthermore, Theorem \ref{thm:lower} implies that for any SFO algorithm $\fA$ and $L, \mu, n, \eps$ such that $L/\mu > \sqrt{10 n}$ and $\eps < 0.003$, there exist a dimension $d = \fO\left( \sqrt{n} L / \mu \log(1/\eps) \right)$ and functions $\{f_i(\vx, \vy)\}_{i=1}^n: \BR^{d} \times \BR^d \to \BR$ which satisfy $\{f_i\}_{i=1}^n$ is $L$-average smooth, $f$ is $(\mu,\mu)$-convex-concave. In order to find an approximate saddle point $(\hat\vx,\hat\vy)$ such that 
\begin{align*}
    \BE\left[\norm{\hat\vx-\vx^*}^2+\norm{\hat\vy-\vy^*}^2\right] \leq \eps,
\end{align*}
algorithm $\fA$ needs at least $\Omega((\sqrt{n}L/\mu)\log(1/\eps))$ steps.

For the case $L / \mu = \fO(\sqrt{n})$, we consider following problem
\begin{align}\label{def:scsc-n}
\begin{split}
\min_{\vx} \max_{\vy} \hat{f} (\vx, \vy) = & \frac{1}{n} \sum_{i=1}^n \hat{f}_i(\vx, \vy) \\
\triangleq & \frac{1}{n}\sum_{i=1}^n \left(\frac{\mu}{2} \norm{\vx}^2 + \frac{\sqrt{n} \hat{L}}{2} (x_i - 1)^2 - \frac{\mu}{2} \norm{\vy}^2 - \frac{\sqrt{n} \hat{L}}{2} (y_i - 1)^2\right), 
\end{split}
\end{align}
where $\hat{f}: \BR^n \times \BR^n \to \BR$, $\hat{L} = \sqrt{\frac{L^2}{2} - \mu^2}$. 

And we provide the lower bound for finding approximate saddle point of Problem (\ref{def:scsc-n}) as follows.
\begin{thm}\label{thm:lower-n}
    For the parameter $L, \mu, n, \eps$ such that $L/\mu > 2$ and $\eps < \frac{1}{8}$, the functions $\hat{f}(\vx, \vy)$ and $\hat{f}_i(\vx, \vy)$ defined in Problem (\ref{def:scsc-n}) satisfy $\hat{f}$ is $(\mu, \mu)$-convex-concave and $\{\hat{f}_i\}_{i=1}^n$ is $L$-average-smooth. Moreover, when we employ any SFO algorithm $\fA$ to solve the Problem (\ref{def:scsc-n}), there holds
    \begin{align*}
        \E \left[ \norm{\vx^{(t)} - \vx^*}^2 + \norm{\vy^{(t)} - \vy^*}^2 \right] > \eps ~~~ \text{ for } t \le n / 2.
    \end{align*}
\end{thm}
\begin{proof}
    It is easily to check that
    \begin{align*}
        \hat{f}(\vx, \vy) = \frac{\mu}{2} \norm{\vx}^2 + \frac{\hat{L}}{2 \sqrt{n}} \norm{\vx - \vone}^2 - \frac{\mu}{2} \norm{\vy}^2 - \frac{\hat{L}}{2 \sqrt{n}} \norm{\vy - \vone}^2.
    \end{align*}
    Hence $\hat{f}$ is $(\mu, \mu)$-convex-concave and the saddle point $(\vx^*, \vy^*)$ of $\hat{f}$ satisfies
    \begin{align*}
        \vx^* = \vy^* = \frac{\hat{L}}{\hat{L} + \sqrt{n} \mu} \vone. 
    \end{align*}
    
    Next, observe that 
    \begin{equation}\label{eq:grad-n}
    \begin{aligned}
        \nabla_{\vx} \hat{f}_i(\vx, \vy) =& \mu \vx + \sqrt{n} \hat{L} (\ve_i^{\top} \vx - 1) \ve_i, \\
        \nabla_{\vy} \hat{f}_i(\vx, \vy) =& - \mu \vy - \sqrt{n} \hat{L} (\ve_i^{\top} \vy - 1) \ve_i.
    \end{aligned}
    \end{equation}
    Then we have
    \begin{align*}
        &\quad \frac{1}{n} \sum_{i=1}^n \left(\norm{\nabla_{\vx} \hat{f}_i(\vx_1, \vy_1) - \nabla_{\vx} \hat{f}_i(\vx_1, \vy_1)}^2 + \norm{\nabla_{\vy} \hat{f}_i(\vx_1, \vy_1) - \nabla_{\vy} \hat{f}_i(\vx_1, \vy_1)}^2\right) \\
        &= \frac{1}{n} \sum_{i=1}^n \left(\norm{\mu(\vx_1 - \vx_2) + \sqrt{n}\hat{L} \ve_i \ve_i^{\top} (\vx_1 - \vx_2) }^2 + \norm{\mu(\vy_1 - \vy_2) + \sqrt{n}\hat{L} \ve_i \ve_i^{\top} (\vy_1 - \vy_2) }^2\right) \\
        &\le \frac{1}{n} \sum_{i=1}^n \left(2 \mu^2 \norm{\vx_1 - \vx_2}^2 + 2 n\hat{L}^2 \norm{\ve_i \ve_i^{\top} (\vx_1 - \vx_2) }^2 + 2 \mu^2 \norm{\vy_1 - \vy_2}^2 + 2 n\hat{L}^2 \norm{\ve_i \ve_i^{\top} (\vy_1 - \vy_2) }^2\right) \\
        &= 2(\hat{L}^2 + \mu^2) \left(\norm{\vx_1 - \vx_2}^2 + \norm{\vy_1 - \vy_2} \right) = L^2 \left(\norm{\vx_1 - \vx_2}^2 + \norm{\vy_1 - \vy_2} \right),
    \end{align*}
    where the first inequality is according to $\norm{\va + \vb}^2 \le 2\norm{\va}^2 + 2\norm{\vb}^2$ and the last equation follows from
    \begin{align*}
        \sum_{i=1}^n \norm{\ve_i \ve_i^{\top} \vu} = \sum_{i=1}^n u_i^2 = \norm{\vu}^2.
    \end{align*}
    Hence, $\{\hat{f}_i\}_{i=1}^n$ is $L$-average-smooth. 
    
    Moreover, by Equation (\ref{eq:grad-n}) and definition of SFO algorithm, we know that
    \begin{align*}
        \vx^{(t)}, \vy^{(t)} \in \spn\left\{ \ve_{i_1}, \dots, \ve_{i_t} \right\}.
    \end{align*}
    Consequently, for $t \le n/2$, there holds
    \begin{align*}
        \E & \left[ \norm{\vx^{(t)} - \vx^*}^2 + \norm{\vy^{(t)} - \vy^*}^2 \right]  \\
        \ge & (n - t)\cdot \frac{\hat{L}^2}{(\hat{L} + \sqrt{n} \mu)^2} 
        \ge \frac{n}{2} \cdot \frac{\hat{L}^2}{2 \hat{L}^2 + 2 n \mu^2} \\
        = & \frac{n}{2} \cdot \frac{L^2 /2 - \mu^2}{L^2 - 2 \mu^2 + 2 n \mu^2} = \frac{n}{4} \left(1 - \frac{2n}{L^2/\mu^2 + 2n - 2}\right) \\
        \ge & \frac{n}{4(n+1)} \ge \frac{1}{8} > \eps,
    \end{align*}
    where we have recalled that $\hat{L}^2 = L^2/2 - \mu^2$, $L/\mu > 2$ and $n \ge 1$.
\end{proof}

Combining Theorem \ref{thm:lower} and \ref{thm:lower-n}, we obtain the result of Theorem \ref{thm:balance-lower-bound}.

%% file: appendix/cc.tex
\section{The Detailed Proof for Extension Case}\label{appendix:csc}

In this section, we provide the detailed proof for the results in Section~\ref{sec:extensions}, including the convex-strongly-concave case and the convex-concave cases.

\subsection{The Proof of Lemma~\ref{lem:csc}}
\begin{proof}
     Just note that 
     \begin{align*}
         \max_{(\vx, \vy) \in \fX \times \fY} |f (\vx, \vy) - f_{\eps, \vx_0} (\vx, \vy)| = \frac{\eps}{4 D_x^2} \max_{\vx \in \fX} \norm{\vx - \vx_0}^2 \le \frac{\eps}{4}.
     \end{align*}
     Therefore, we can conclude that 
     \begin{align*}
         \max_{\vy \in \fY} f (\hat{\vx}, \vy) - \min_{\vx \in \fX} f (\vx, \hat{\vy}) &= f(\hat{\vx}, \vy^*_f(\hat{\vx})) - f(\vx^*_f(\hat{\vy}), \hat\vy) \\
         &\le \left(f_{\eps, \vx_0}(\hat{\vx}, \vy^*_f(\hat{\vx})) + \frac{\eps}{4}\right) - \left(f_{\eps, \vx_0}(\vx^*_f(\hat{\vy}), \hat\vy) - \frac{\eps}{4}\right) \\
         &\le \frac{\eps}{2} + \max_{\vy \in \fY} f_{\eps, \vx_0} (\hat{\vx}, \vy) - \min_{\vx \in \fX} f_{\eps, \vx_0} (\vx, \hat{\vy}),
     \end{align*}\
     where $\vy^*_f(\hat{\vx}) = \argmax_{\vy \in \fY} f (\hat{\vx}, \vy)$ and $\vx^*_f(\hat{\vy}) = \argmin_{\vx \in \fX} f (\vx, \hat{\vy})$.
\end{proof}

\subsection{The Proof of Corollary~\ref{cor:csc}}
\begin{proof}
    By the definition of $f_{\eps, \vx_0}$, we know that $f_{\eps, \vx_0}$ is $\left( \frac{\eps}{4 D_x^2}, \mu_y \right)$-convex-concave. 
    The smoothness of $f_{\eps, \vx_0}$ can be verified by
    \begin{align*}
        &\quad \frac{1}{n} \sum_{i=1}^n \norm{\nabla f_{\eps, \vx_0, i}(\vx, \vy) - \nabla f_{\eps, \vx_0, i}(\vx', \vy')}^2 \\
        &= \frac{1}{n} \sum_{i=1}^n \norm{\nabla f_{i}(\vx, \vy) - \nabla f_{i}(\vx', \vy') + \frac{\eps}{4 D_x^2}(\vx - \vx')}^2 \\
        &\le \frac{2}{n} \sum_{i=1}^n \norm{\nabla f_{i}(\vx, \vy) - \nabla f_{i}(\vx', \vy')}^2 + \frac{\eps^2}{8 D_x^4} \norm{(\vx - \vx')}^2 \\
        &\le 2 L^2 \left( \norm{\vx - \vx'}^2 + \norm{\vy - \vy'}^2 \right) + \frac{\eps^2}{8 D_x^4} \norm{\vx - \vx'}^2 \\
        &\le 4 L^2 \left( \norm{\vx - \vx'}^2 + \norm{\vy - \vy'}^2 \right),
    \end{align*}
    where the first inequality is according to $(a + b)^2 \le 2a^2 + 2b^2$ and the second inequality follows from the smoothness of $\{f_i\}_{i=1}^n$ and $\eps \le 4 L D_x^2$. 
    Hence, $\{f_{\eps, \vx_0, i}\}_{i=1}^n$ is $2 L$-average smooth. 
    
    Then by Corollary \ref{cor:ASVRE}, the number of SFO calls of AL-SVRE for finding $\eps/2$-saddle point of $f_{\eps, \vx_0}$ corresponds to finding an $\eps$-saddle point of $f$, that is 
    \begin{align*}
        \tilde\fO \left( \sqrt{n\left(\sqrt{n} + \frac{2 L}{\frac{\eps}{4 D_x^2}}\right)\left(\sqrt{n} + \frac{2 L}{\mu_y}\right)} \right) = \tilde\fO\left(\left(n + D_x \sqrt{\frac{n L \kappa_y}{\eps}} + n^{3/4}\sqrt{\kappa_y} + n^{3/4}D_x \sqrt{\frac{L}{\eps}}\right)\right).
    \end{align*}
\end{proof}

\subsection{The Proof of Lemma~\ref{lem:cc}}
\begin{proof}
     Note that 
     \begin{align*}
         \max_{(\vx, \vy) \in \fX \times \fY} |f (\vx, \vy) - f_{\eps, \vx_0, \vy_0} (\vx, \vy)| =  \max_{(\vx, \vy) \in \fX \times \fY} \left\vert \frac{\eps}{8 D_x^2} \norm{\vx - \vx_0}^2 - \frac{\eps}{8 D_y^2} \norm{\vy - \vy_0}^2 \right\vert \le \frac{\eps}{4}.
     \end{align*}
    For the same reason we got the result in the proof of Lemma \ref{lem:csc}, we also have
     \begin{align*}
         \max_{\vy \in \fY} f (\hat{\vx}, \vy) - \min_{\vx \in \fX} f (\vx, \hat{\vy}) &= f(\hat{\vx}, \vy^*_f(\hat{\vx})) - f(\vx^*_f(\hat{\vy}), \hat\vy) \\
         &\le \left(f_{\eps, \vx_0}(\hat{\vx}, \vy^*_f(\hat{\vx})) + \frac{\eps}{4}\right) - \left(f_{\eps, \vx_0}(\vx^*_f(\hat{\vy}), \hat\vy) - \frac{\eps}{4}\right) \\
         &\le \frac{\eps}{2} + \max_{\vy \in \fY} f_{\eps, \vx_0} (\hat{\vx}, \vy) - \min_{\vx \in \fX} f_{\eps, \vx_0} (\vx, \hat{\vy}),
     \end{align*}\
     where $\vy^*_f(\hat{\vx}) = \argmax_{\vy \in \fY} f (\hat{\vx}, \vy)$ and $\vx^*_f(\hat{\vy}) = \argmin_{\vx \in \fX} f (\vx, \hat{\vy})$.
\end{proof}

\subsection{The Proof of Corollary~\ref{cor:cc}}
\begin{proof}
    By the definition of $f_{\eps, \vx_0, \vy_0}$, we know that $f_{\eps, \vx_0, \vy_0}$ is $\left( \frac{\eps}{8 D_x^2}, \frac{\eps}{8 D_y^2} \right)$-convex-concave. 
    The smoothness of $f_{\eps, \vx_0}$ can be verified by
    \begin{align*}
        &\quad \frac{1}{n} \sum_{i=1}^n \norm{\nabla f_{\eps, \vx_0, \vy_0, i}(\vx, \vy) - \nabla f_{\eps, \vx_0, \vy_0, i}(\vx', \vy')}^2 \\
        &= \frac{1}{n} \sum_{i=1}^n \norm{\nabla f_{i}(\vx, \vy) - \nabla f_{i}(\vx', \vy') + \frac{\eps}{8 D_x^2}(\vx - \vx') - \frac{\eps}{8 D_y^2}(\vy - \vy')}^2 \\
        &\le \frac{3}{n} \sum_{i=1}^n \norm{\nabla f_{i}(\vx, \vy) - \nabla f_{i}(\vx', \vy')}^2 + \frac{3 \eps^2}{64 D_x^4} \norm{(\vx - \vx')}^2 + \frac{3 \eps^2}{64 D_y^4} \norm{(\vy - \vy')}^2 \\
        &\le 4 L^2 \left( \norm{\vx - \vx'}^2 + \norm{\vy - \vy'}^2 \right),
    \end{align*}
    where the first inequality is according to $(a + b + c)^2 \le 3a^2 + 3b^2 + 3c^2$ and the second inequality follows from the smoothness of $\{f_i\}_{i=1}^n$ and $\eps \le 4 L \min\{D_x^2, D_y^2\}$. 
    Hence, $\{f_{\eps, \vx_0, i}\}_{i=1}^n$ is $2 L$-average smooth. 
    
    Then by Corollary \ref{cor:ASVRE}, the number of SFO calls of AL-SVRE for finding $\eps/2$-saddle point of $f_{\eps, \vx_0}$ corresponds to finding an $\eps$-saddle point of $f$, is 
    \begin{align*}
        \tilde\fO \left( \sqrt{n\left(\sqrt{n} + \frac{2 L}{\frac{\eps}{8 D_x^2}}\right)\left(\sqrt{n} + \frac{2 L}{\frac{\eps}{8 D_y^2}}\right)} \right) = \tilde\fO\left(\left(n + \frac{\sqrt{n} L D_x D_y}{\eps} + n^{3/4} (D_x + D_y) \sqrt{\frac{L}{\eps}}\right)\right).
    \end{align*}
\end{proof}